\documentclass[10pt,a4paper]{article} 
\usepackage{amsmath,amssymb,amsthm,amsfonts,amscd,euscript,verbatim, t1enc, newlfont, graphicx, cancel}
\usepackage{hyperref}
\usepackage[all]{xy}
\providecommand{\keywords}[1]{\textbf{\textit{Key words and phrases }} #1}
\providecommand{\subjclass}[1]{\textbf{\textit{2010 Mathematics Subject Classification.}} #1}

\theoremstyle{definition}
\newtheorem{theo}{Theorem}[subsection]

\newtheorem{theore}{Theorem}[section]
\newtheorem{pr}[theo]{Proposition}

 \newtheorem{lem}[theo]{Lemma}
 \newtheorem{llem}[theore]{Lemma}
 \newtheorem{coro}[theo]{Corollary}

\theoremstyle{remark}
\newtheorem{rema}[theo]{Remark}

\newtheorem{rrema}[theore]{Remark}
\theoremstyle{definition}
\newtheorem{defi}[theo]{Definition}

\newtheorem{prop}[theore]{Proposition}

\numberwithin{equation}{subsection}

\newcommand\cu{\underline{C}}
\newcommand\cupr{\underline{C}'}
\newcommand\cuz{\underline{C}_0}

\newcommand\du{\underline{D}}
\newcommand\eu{\underline{E}}
\newcommand\au{\underline{A}}

\newcommand\bu{\underline{B}}
\newcommand\hu{\underline{H}}

\newcommand\obj{\operatorname{Obj}}
\newcommand\mo{\operatorname{Mor}}
\newcommand\id{\operatorname{id}}
\DeclareMathOperator\adfu{\operatorname{AddFun}}

\DeclareMathOperator\kar{\operatorname{Kar}}
 \DeclareMathOperator\ke{\operatorname{Ker}}
 \DeclareMathOperator\cok{\operatorname{Coker}}
\DeclareMathOperator\imm{\operatorname{Im}}
\DeclareMathOperator\co{\operatorname{Cone}}

\DeclareMathOperator\prli{\varprojlim}
\DeclareMathOperator\inli{\varinjlim}
\DeclareMathOperator\hinli{\underrightarrow{\operatorname{hocolim}}}

\newcommand\hw{{\underline{Hw}}}
\newcommand\hrt{{\underline{Ht}}}

\newcommand\alz{{\aleph_0}}

\newcommand\cocp{\underline{\coprod}\cp}
\newcommand\cocpp{{\underline{\coprod}\cp'}}
\newcommand\hatc{\operatorname{Coh}_{\cp'}}
\newcommand\lo{\mathcal{LO}}
\newcommand\ro{\mathcal{RO}}
\newcommand\hdu{\widehat{H}}

\newcommand\wstu{w^{st}}

\newcommand\modd{\operatorname{Mod}}

\newcommand\q{{\mathbb{Q}}}

\newcommand\z{{\mathbb{Z}}}

 \newcommand\lan{\langle}
\newcommand\ra{\rangle}

\newcommand\ob{^{-1}}

\newcommand\al{\alpha}
\newcommand\be{\beta}

\newcommand\ns{\{0\}}
\newcommand\ab{\operatorname{Ab}}

\newcommand\cq{\mathcal{Q}}
\newcommand\ooo{\mathcal{O}}
\newcommand\cp{\mathcal{P}}
\newcommand\cpt{{\tilde{\mathcal{P}}}}
\newcommand\cpn{\mathcal{P}-\mathbf{null}}
\newcommand\cpcn{\mathcal{P}'-\mathbf{conull}}
\newcommand\perpp{{}^{\perp}}
\newcommand\opp{^{op}}

\newcommand\lscp{L_s\mathcal{P}}

\begin{document}

\title{On perfectly generated weight structures and  adjacent $t$-structures}
\author{Mikhail V. Bondarko
   \thanks{ 
 \S\S\ref{sts}--\ref{swgs}  of paper were   supported by the Russian Science Foundation grant no. 16-11-00073
.  Results of section \S\ref{sdeg}  were supported the Russian Science Foundation grant no.  20-41-04401.}}\maketitle
\begin{abstract} 
This paper is dedicated to the study of {\it smashing} weight structures 
(these are the weight structures "coherent with coproducts"), and the application of their properties to $t$-structures. In particular, we prove that hearts of compactly generated $t$-structures are Grothendieck abelian categories; this statement strengthens earlier results of several other authors.

The central theorem of the paper is as follows: any perfect (as defined by Neeman) set of objects of a triangulated category generates a weight structure; we say that weight structures obtained this way are  perfectly generated. An important family of perfectly generated weight structures are (the opposites to) the ones {\it right adjacent} to compactly generated $t$-structures; they give injective cogenerators for the hearts of the latter. Moreover, we establish the following not so explicit result: any smashing weight structure on a well generated triangulated category (this is a generalization of the notion of a compactly generated category that was also defined by Neeman) is perfectly  generated; actually, we  prove more than that.

Furthermore, we give a classification of compactly generated {\it torsion theories} (these generalize both weight structures and $t$-structures) that extends the corresponding result of D. Pospisil  and J.  \v{S}\v{t}ov\'\i\v{c}ek to arbitrary smashing triangulated categories. This gives a generalization of a $t$-structure statement due to B. Keller and P. Nicolas.

\end{abstract}


\tableofcontents
 \section*{Introduction}
The main subject of this paper are smashing weight structures and the application of their properties to the study of $t$-structures.

We recall that weight structures are defined somewhat similarly to $t$-structures; yet their properties are quite distinct. A weight structure on a triangulated category $\cu$ is a couple $w=(\cu_{w\le 0},\cu_{w\ge 0})$ of classes of its objects, subject to certain axioms. One says that $w$ is smashing if  $\cu$ is (that is, $\cu$ is closed with respect to small coproducts) and $\cu_{w\ge 0}$ is closed with respect to $\cu$-coproducts; note that $\cu_{w\le 0}$ is closed with respect to $\cu$-coproducts automatically.

Let us adopt the following convention: for 
 $S\subset \obj\cu$ we will write $S\perpp$ (resp. $\perpp S$) for the class of those $M\in \obj \cu$ such that the morphism group $\cu(N,M)$ (resp. $\cu(M,N)$) is zero   for all $N\in S$. Then the main existence of weight structures result of this paper can be formulated as follows. 

\begin{theore}\label{textw}
Assume that $\cu$ 
 smashing; let $\cp$ be a {\it perfect} set of objects of $\cu$ (i.e., 
 $\cpn$ is closed with respect to  coproducts, where $\cpn$ is  the class 
of  $h\in \mo(\cu)$ 
 such that $\cu(P,h)=0$ for all $P\in \cp$).

Then 
$w=(L,R)$ is a  smashing weight structure, 
 where $R=\cap_{i<0}(\cp\perpp[i])$ and $L=(\perpp R)[1]$. 

Moreover,  the class $L$ may be described "more explicitly" in terms of $\cp$; cf. Theorem \ref{tpgws} 
  below. \end{theore}
	
	This result significantly generalizes Theorem 5 of \cite{paucomp}, where all the elements of $\cp$ were assumed to be {\it compact}, that is, for any $P\in \cp$ 
	the functor $\cu(P,-):\cu\to \ab$ respects coproducts. 

We also give some applications of this existence statement. The most important of them treats {\it compactly generated} $t$-structures. 
 Note that these are popular objects of study (ever since their introduction in \cite{talosa}), with plenty of examples important to various areas of mathematics.

Recall that a $t$-structure $t=(\cu_{t\le 0},\cu_{t\ge 0})$ on $\cu$ is {\it generated} by a class $\cp\subset \obj \cu$ whenever $\cu_{t\le 0}=\cap_{i\ge 1}(\cp\perpp[i])$. So, we prove the following statement (see Corollary \ref{cgroth} below; actually, we start the main body of the paper from reducing this result to the existence of an injective cogenerator in $\hrt$).

 \begin{theore}\label{tgroth}
 Assume that $t$ is a $t$-structure on a smashing triangulated category $\cu$, and $t$ is generated by a set $\cp$ of compact objects.

Then the  heart $\hrt$ of $t$ is a Grothendieck abelian category, and the zeroth $t$-homology of the objects of 
 any full  triangulated subcategory $\cu_0$ of $\cu$ containing $\cp$ give generators for $\hrt$.\end{theore}

We recall here that statements of this sort are quite popular in the literature; see the introduction to \cite{saostovi}. In particular,  Theorem 3.7 of   \cite{parrasao} says that   countable colimits in the 	 category $\hrt$ are exact  for any compactly generated $t$-structure $t$ (this is clearly weaker then being a Grothendieck abelian category). Moreover, in \cite{humavit} our theorem was proved in the case where $\cu$ is an {\it algebraic} triangulated category and $t$ is non-degenerate, whereas  in \cite{saostovi} 
	 it was proved under the assumption that 
	$\cu$ is a {\it topological} well generated  (see Proposition \ref{pwgwstr} below) category.\footnote{Note also that Theorem B of ibid. says that $\hrt$ is an AB5 abelian category whenever $\cu$ is a "strong stable derivator" triangulated category, whereas Theorem C of ibid. gives the existence of generators for a wide class of $t$-structures.} Furthermore, a proof of 
	 the general case of Theorem \ref{tgroth} that relies on arguments 
	  different from our ones 
	 was 
	  independently  obtained in \cite{saostov}. 

The proof of Theorem \ref{tgroth} relies on two "recent" prerequisites. The first of them is the existence of a weight structure that is {\it right adjacent} to $t$ (that is, $\cu_{w\le 0}=\cu_{t\le 0}$); it is an easy consequence of Theorem \ref{textw} (along with certain results of earlier texts of the author). We use a "cogenerator" of  the heart $\hw$ to construct an injective cogenerator of the category $\hrt$. This enables us to apply Theorem 3.3 of \cite{posistov} (this is our second prerequisite) to obtain that $\hrt$ is an AB5 abelian category. Alternatively, if $t$ is non-degenerate then one may argue similarly to the proof of \cite[Corollary 4.9]{humavit}; see Remark \ref{r438} below or Corollary 4.3.9(3) of \cite{bpure}.\footnote{Respectively, loc. cit. is just a little weaker than Theorem \ref{tgroth}. 
Note also that \cite{bpure} is much more self-contained than our current paper. 
 Respectively, ibid. is quite long and and rather difficult to read. 
 For this reason the author has decided to split it and publish the resulting texts separately (see Remark 0.5 of ibid.); note also that these newer texts contain some results not contained in ibid., and the exposition in them is more accurate.}  

We also prove the existence of a certain "join" operation on the class of 
 perfectly generated weight structures on a given triangulated category $\cu$; see Corollary \ref{cwftw}(2) and Remark \ref{revenmorews}. 

Another application of Theorem  \ref{textw} is the following "well generatedness" result for weight structures  (saying in particular that all smashing weight structures on well generated categories can be obtained from  that theorem). 

\begin{prop}\label{pwgwstr}
Assume that $\cu$ is a well generated triangulated category (i.e., there exists a {\it regular} cardinal $\al$ 
and a perfect set  $S$ of {\it $\al$-small} objects such that $S\perpp=\ns$; see Definition \ref{dbecomp}). 

Then for any smashing weight structure $w$  
  on $\cu$ there exists a  cardinal $\al'$ 
	 such that for any regular $\be\ge \al'$ the weight structure $w$ is strongly $\be$-well generated in the following sense: the couple $(\cu_{w\le 0}\cap \obj \cu^\be, \cu_{w\ge 0}\cap \obj \cu^\be)$ is a weight structure on the triangulated subcategory $\cu^\be$ of $\cu$ consisting of $\be$-compact objects (see Definition \ref{dbecomp}(\ref{idcomp})), the class $\cp=\cu_{w\le 0}\cap \obj \cu^\be$ is essentially small and perfect, and $w=(L,R)$, where  $R=(\cp\perpp)[-1]$ and $L=(\perpp R)[1]$ (cf. Theorem \ref{textw}).  
\end{prop}

The proof of this statement is closely related to  {\it torsion theories}. 
We recall that torsion theories essentially generalize both weight structures and $t$-structures. 
   Respectively, our classification of compactly generated torsion theories (in Theorem \ref{tclass})  immediately 
 gives the corresponding classifications of compactly generated weight structures and $t$-structures. All of these statements generalize the corresponding results of \cite{postov} 
 where {\it torsion theories} were studied.\footnote{Actually, in ibid. the term {\it complete Hom-orthogonal pair} was used. In some other papers torsion theories are called torsion pairs.} 

\begin{rrema}
 The existence of a torsion theory generated by a set $\cp$  of compact objects of $\cu$  (see Definition \ref{dhop}(\ref{ittg}) and  Theorem \ref{tclass}(\ref{iclass1}) below) is provided by Theorem 4.3 of \cite{aiya}; cf. Corollary 4.6 of ibid. for the case of $t$-structures and weight structures.  
 Next, Theorem 3.7 of \cite{postov} gave a certain classification of torsion theories of this type 
  when  $\cu$ is a "stable derivator" category. In  Theorem \ref{tclass}(\ref{iclass2},\ref{iclassts}) below we drop this assumption.
 
 Moreover, applying 
  Theorem \ref{tclass}(\ref{iclassts})  to $t$-structures we obtain the corresponding generalization of \cite[Theorem A.9]{kellerw}. 
  On the other hand, no analogue of   Theorem \ref{textw} is currently known to hold for $t$-structures; see Remark \ref{rigid}(1) below. Consequently, the author does not know whether arbitrary perfect sets of objects generate torsion theories. 
\end{rrema}

Let us  now describe the contents  of the paper. Some more information of this sort may be found in the beginnings of sections. 

In \S\ref{sts} we study $t$-structures. Applying  Theorem 3.3 of \cite{posistov} along with properties of certain Kan extensions (taken from \cite{krause})  we prove that the heart of a compactly generated $t$-structure is a Grothendieck abelian category whenever this category has an injective cogenerator.

In \S\ref{spgws} we switch to weight structures. Using rather standard countable homotopy colimit arguments we prove that any perfect set of objects generates a (smashing) weight structure; we also study the heart of this weight structure. Our main examples to this statement give weight structures that are right adjacent to compactly generated $t$-structures; their properties enable us to prove that injective cogenerators for the hearts of the latter exist indeed.

In \S\ref{swgs} we study  torsion theories; these essentially generalize both weight structures and $t$-structures. Respectively, our classification of compactly generated torsion theories gives a certain classification of compactly generated weight structures and $t$-structures on a given category. 
  Moreover, we study smashing torsion theories in well generated triangulated categories; this enables us to prove  Proposition \ref{pwgwstr}. 

In \S\ref{sdeg} we discuss some more general conditions that ensure the existence of adjacent $t-$ and weight structures

The author is deeply grateful to  Prof. Manuel Saorin,  Prof. George-Ciprian Modoi, and to the referees for their very useful comments.

\section{
On  hearts of compactly generated $t$-structures}\label{sts}

In  \S\ref{snotata} we give some definitions and conventions related to (mostly) triangulated categories.

In \S\ref{scgts} we recall some basic on $t$-structures (and on generators for them).

In \S\ref{stab5} we describe some more definitions and properties of  $t$-structures; they allow us to reduce the statement that hearts of compactly generated $t$-structures are AB5 categories to Corollary \ref{csymt} below.

In \S\ref{scoext} we recall (from \cite[\S2]{krause}) some properties of   left Kan extensions of homological functors defined on certain triangulated subcategories of $\cu$ to $\cu$ itself. We use them to prove that the  heart of  a compactly generated $t$-structure is a Grothendieck abelian category whenever it is AB5. 

\subsection{Some definitions and notation  for triangulated categories}\label{snotata}

\begin{itemize}

\item All products and coproducts in this paper will be small.

\item Given a category $C$ and  $X,Y\in\obj C$  we will write
$C(X,Y)$ for  the set of morphisms from $X$ to $Y$ in $C$.

\item For categories $C'$ and $C$ we write $C'\subset C$ if $C'$ is a full 
subcategory of $C$.

\item Given a category $C$ and  $X,Y\in\obj C$, we say that $X$ is a {\it
retract} of $Y$ 
 if $\id_X$ can be 
 factored through $Y$.\footnote{Clearly,  if $C$ is triangulated or abelian, 
then $X$ is a retract of $Y$ if and only if $X$ is its direct summand.}\ 

\item A 
 subcategory $\hu$ of an additive category $C$ 
is said to be {\it retraction-closed} in $C$ if it contains all retracts of its objects in $C$.

\item The symbol $\cu$ below will always denote some triangulated category;  it will often be endowed with a weight structure $w$. The symbols $\cu'$ and $\du$ will  also be used  for triangulated categories only.

\item For any  $A,B,C \in \obj\cu$ we will say that $C$ is an {\it extension} of $B$ by $A$ if there exists a distinguished triangle $A \to C \to B \to A[1]$.

\item A class $\cp\subset \obj \cu$ is said to be  {\it extension-closed}
    if it 
		is closed with respect to extensions and contains $0$.  




\item The smallest 
 retraction-closed extension-closed class of objects of $\cu$ containing   $\cp$  will be called the {\it 
envelope} of $\cp$.  

\item For $X,Y\in \obj \cu$ we will write $X\perp Y$ if $\cu(X,Y)=\ns$. 

For
$D,E\subset \obj \cu$ we write $D\perp E$ if $X\perp Y$ for all $X\in D,\
Y\in E$.

Given $D\subset\obj \cu$ we  will write $D^\perp$ for the class
$$\{Y\in \obj \cu:\ X\perp Y\ \forall X\in D\}.$$
Dually, ${}^\perp{}D$ is the class
$\{Y\in \obj \cu:\ Y\perp X\ \forall X\in D\}$.

\item Let $\cu'$ be a full triangulated subcategory of $\cu$. Then we will say that the elements of $\obj \cupr{}\perpp\subset \obj \cu$ are {\it $\cu'$-local}.

\item For a morphism $f\in\cu (X,Y)$ (where $X,Y\in\obj\cu$) we will call the third vertex
of (any) distinguished triangle $X\stackrel{f}{\to}Y\to Z$ a {\it cone} of
$f$.\footnote{Recall 
that different choices of cones are connected by non-unique isomorphisms.}\

\item Below $\au$ will always  denote some abelian category. 

	\item We will say that an additive covariant (resp. contravariant) functor from $\cu$ into $\au$ is {\it homological} (resp. {\it cohomological}) if it converts distinguished triangles into long exact sequences.
\end{itemize}


We will sometimes need the following simple observation.

\begin{lem}\label{lcloc}
Let $\cu'$ be a full triangulated subcategory of $\cu$. Then the full subcategory of $\cu'$-local objects of $\cu$ is triangulated. 
\end{lem}
\begin{proof} Obvious and cointained in Lemma 9.1.12 of \cite{neebook}.
\end{proof}

\subsection{
A reminder on $t$-structures}\label{scgts}

Let us now recall the notion of  a $t$-structure (mainly to fix  notation). 

\begin{defi}\label{dtstr}

A couple of subclasses  $(\cu_{t\le 0},\cu_{t\ge 0})$ of $\obj\cu$ will be said to be a
$t$-structure $t$ on $\cu$  if 
they satisfy the following conditions:

(i) $\cu_{t\le 0}$ and $\cu_{t\ge 0}$  are strict, i.e., contain all
objects of $\cu$ isomorphic to their elements.

(ii) $\cu_{t\le 0}\subset \cu_{t\le 0}[1]$ and $\cu_{t\ge 0}[1]\subset \cu_{t\ge 0}$.

(iii)  $\cu_{t\ge 0}[1]\perp \cu_{t\le 0}$.

(iv) For any $M\in\obj \cu$ there exists a  {\it $t$-decomposition} distinguished triangle
\begin{equation}\label{tdec}
L_tM\to M\to R_tM{\to} L_tM[1]
\end{equation} such that $L_tM\in \cu_{t\ge 0}, R_tM\in \cu_{t\le 0}[-1]$.

2. $\hrt$ is the full subcategory of $\cu$ whose object class is $\cu_{t=0}=\cu_{t\le 0}\cap \cu_{t\ge 0}$.
\end{defi}

We will also give some auxiliary definitions.

\begin{defi}\label{dtstro}
1. For any $i\in \z$ we will use the notation $\cu_{t\le i}$ (resp. $\cu_{t\ge i}$) for the class $\cu_{t\le 0}[i]$ (resp. $\cu_{t\ge 0}[i]$). 

2. $\hrt$ is the full subcategory of $\cu$ whose object class is $\cu_{t=0}=\cu_{t\le 0}\cap \cu_{t\ge 0}$.

3. We will say that $t$ is {\it left (resp. right) non-degenerate} if $\cap_{i\in \z}\cu_{t\ge i}=\ns$ (resp. $\cap_{i\in \z}\cu_{t\le i}=\ns$). 

Moreover, 
 $t$ is said to be {\it non-degenerate} if it is both left and right non-degenerate. 

4. We  say that $t$ is  {\it generated}  by a class  $\cp\subset \cu$   whenever $\cu_{t\le 0}=(\cup_{i>0}\cp[i])\perpp$.
\end{defi}


Let us recall some well-known properties of $t$-structures. 

\begin{pr}\label{prtst}
Let $t$ be a $t$-structure on a triangulated category $\cu$. Then the following statements are valid.

\begin{enumerate}
\item\label{itcan}
The triangle (\ref{tdec}) is canonically and functorially determined by $M$. Moreover, $L_t$ is right adjoint to the embedding $ \cu_{t\ge 0}\to \cu$ (if we consider $ \cu_{t\ge 0}$ as a full subcategory of $\cu$) and $R_t$ is left adjoint to  the embedding $ \cu_{t\le -1}\to \cu$.

\item\label{itha}
$\hrt$ is 
  an abelian category with short exact sequences corresponding to distinguished triangles in $\cu$.

\item\label{itho}
For any $n\in \z$ we will use the notation $t_{\ge n}$ for the functor $[n]\circ L_t\circ [-n]$, and $t_{\le n}=[n+1]\circ R_t\circ [-n-1]$.

Then there is a canonical isomorphism of functors $t_{\le 0}\circ t_{\ge 0}\cong t_{\ge 0}\circ  t_{\le 0}$. 
  (if we consider these functors as endofunctors of $\cu$), and the composite functor $H^t=H_0^t$ actually takes values in the subcategory $\hrt$ of $ \cu$. Furthermore, this functor $H^t:\cu \to \hrt$   is homological.  

\item\label{itd} 
For any $M\in \cu_{t\ge 0}$ there exists a (canonical) distinguished triangle 
$t_{\ge 1}(M)\to M\to H^t_0(M)\to t_{\ge 1}(M)[1]$. Respectively, $M$ 
 belongs to $M\in \cu_{t\ge 1}$ if and only if $H^t_0(M)=0$.

\item\label{itperp} $\cu_{t\le 0}= \cu_{t\ge 1}\perpp$ and $\cu_{t\ge 0}=(\cu_{t\le- 1}^{\perp})$; hence these classes are retraction-closed and extension-closed in $\cu$.

\item\label{itcone} For $M,N\in \cu_{t\le 0}$ and $f\in \cu(M,N)$ the object $\co(f)$ belongs to $\cu_{t\le 0}$ as well if and only if the morphism $H_0^t(f)$ is monomorphic in $\hrt$.

\item\label{itdeg} Assume that $\cu$ is a full subcategory of a triangulated category $\cu'$ and for the (identical) embedding $\cu\to \cu'$ there exists a right adjoint. Then there exists a unique $t$-structure $t'$ on $\cu'$ such that $\cu_{t\le 0}=\cu'_{t'\le 0}$, and  we also have $\cu'_{t'=0}=\cu_{t=0}$. 
 Moreover, if $t$ is generated by a  class $\cp\subset \obj \cu$ (in $\cu$)  then $t'$ is  generated by  $\cp$ in the category $\cu'$. 

\end{enumerate}
\end{pr}
\begin{proof}
All of these statements except the 
 two last ones were essentially established in \S1.3 of \cite{bbd} (yet see Remark \ref{rtst}(4) below). 

To prove assertion \ref{itcone} we note that     $\co(f)$ belongs to $\cu_{t\le 1}$ and $H_1^t(N)=0$ according to assertion \ref{itperp}.
Hence assertion \ref{itho} gives the following long exact sequence in the category $\hrt$: 
$$\dots\to  0=H_1^t(N)\to  H_1^t(\co(f)) \to H_0^t(M)\stackrel{H_0^t(f)}{\longrightarrow} H_0^t(N) \to \dots$$ 
Along with the last statement in assertion \ref{itd} this yields the result.

Assertion \ref{itdeg} 
 immediately follows from Proposition 
  3.4(1,3) of \cite{bvt}. 
\end{proof}

\begin{rema}\label{rtst}

1. The notion of a $t$-structure is clearly self-dual, that is, the couple $t\opp=(\cu_{t\ge 0},\cu_{t\le 0})$ gives a $t$-structure on the category $\cu\opp$. We will say that the latter $t$-structure is {\it opposite} to $t$.   

2. Part \ref{itperp} of our proposition says that  $t$ is  generated by  
 $\cu_{t\ge 0}$; 
 moreover, this class (along with its shifts) is closed with respect to coproducts.

3. We also obtain that the couple $t$ is uniquely determined by the choice either of $\cu_{t\ge 0}$ or of  $\cu_{t\le 0}$. 
Hence there can exist at most one $t$-structure that is generated by a given class of objects of $\cu$. 

4. Even though in \cite{bbd} where $t$-structures were introduced  and in several preceding papers of the author the "cohomological convention" for $t$-structures was used, in the current text we 
 use the homological convention; the reason for this is that it is coherent with the homological convention for weight structures (see Remark \ref{rstws}(3) below). Respectively, 
 our notation $\cu_{t\ge 0}$ 
 corresponds to the class $\cu^{t\le 0}$ in the cohomological convention. \end{rema}

\subsection{On smashing categories, compactly generated $t$-structures, and their hearts}
\label{stab5}

We will also need a few definitions related to 
 infinite (co)products.

\begin{defi}\label{dsmash}
\begin{enumerate}
\item\label{ismcat}
 We will say that a triangulated category $\cu$  is {\it (co)smashing} if it is closed with respect to (small) 
 coproducts (resp., products).

\item\label{ismcl}
If $\cu$ is (co)smashing and $\cp$ is a class of objects of $\cu$ then 
 $\cp$ is said to be  {\it (co)smashing} (in $\cu$) if it is closed with respect to $\cu$-coproducts (resp., $\cu$-products).

\item\label{idloc}
 If $\cu$ is smashing and $\du$ is a triangulated subcategory of $\cu$ that may be equal to $\cu$,   one says that 
	$\cp$ generates $\du$ {\it as a localizing subcategory} of $\cu$ if $\du$ is the smallest strictly 
	full 	 triangulated subcategory of $\cu$ that contains $\cp$ 
and is closed with respect to  $\cu$-coproducts.

\item\label{idcc} 
It will be convenient for us to use the following somewhat clumsy  terminology: a homological functor $H:\cu\to \au$ (where $\au$ is an abelian category) will be called a {\it cc} (resp. {\it wcc}) functor if it respects all  coproducts (resp. countable coproducts, i.e., the image of any countable coproduct diagram in $\cu$ is the corresponding coproduct diagram in $\au$), 
whereas a cohomological functor  $H'$ from $\cu$ into $\au$ will be called a {\it cp} functor if it converts all (small) coproducts that exist in $\cu$ into the corresponding $\au$-products. 

\item\label{idbrown} We will say that a smashing category $\cu$  satisfies the  {\it Brown representability} property whenever any 
cp functor from $\cu$ into 
 abelian groups is representable. 

\item\label{icomp}
An object $M$ of a smashing category $\cu$ is said to be {\it compact} if 
 the functor $H^M=\cu(M,-):\cu\to \ab$ respects coproducts. 

We will 
write $\cu^{\alz}$ for the full subcategory of $\cu$ whose objects are the compact objects of $\cu$; note that $\cu^{\alz}$ is 
triangulated according to (the easy) Lemma 4.1.4 of \cite{neebook}. 

\item\label{icompgc}
We say that $\cp$ {\it compactly generates} (a smashing category) $\cu$ and that $\cu$ is compactly generated if $\cp$ generates $\cu$ as its own localizing subcategory  and $\cp$ is a {\bf set} of compact objects of $\cu$.

\item\label{ismt} 
 A  $t$-structure $t$ on $\cu$ is said to be  (co)smashing if $\cu$ is (co)smashing and the class $\cu_{t\le 0}$ is 
 smashing (resp., $\cu_{t\ge 0}$ is 
 cosmashing). 

\item\label{icompgt} We will say that $t$ as above is {\it compactly generated} (by $\cp\subset \obj \cu$) if $\cp$ is a {\bf set} of compact objects.
\end{enumerate}
\end{defi}

Let us prove 
 easy properties of these notions.


\begin{pr}\label{pcgt}
Assume that $t$ is a 
$t$-structure on a smashing  triangulated category $\cu$, and $\cp$ is a { set} of compact objects of $\cu$. Then the following statements are valid. 

I. 
 Assume in addition that $\cu$ is cosmashing.

1. Then the class $\cu_{t\le 0}$ is cosmashing in $\cu$.

2. The category $\hrt$ is closed with respect to small products, and for $A_i\in \obj \hrt$ we have $\prod_{\hrt}A_i\cong H_0^t(\prod_{\cu}A_i)$.

3. The product of any family of distinguished triangles in $\cu$ is also distinguished. 


II. Assume that $t$ is a compactly generated $t$-structure. Then $t$ is smashing.

III. Assume that $t$ is smashing.

1. Then the category $\hrt$ is closed with respect to (small) coproducts and the embedding $\hrt\to \cu$ respects coproducts.

2. The functors $t_{\le 0}$, $t_{\ge 0}$, and $H_0^t$ respect $\cu$-coproducts. 

IV. 
 $\cp$ generates a certain $t$-structure on $\cu$.

V. Assume that $\cu$ is compactly generated. 

1. Then $\cu$ is cosmashing and  
 satisfies the Brown representability property. 

2. If $F: \cu\to \du$ is an exact functor (between triangulated categories) that respects coproducts then it possesses 
 a right adjoint $F^*$.

Moreover, $F^*$ respects coproducts as well whenever $F$ is a full embedding and the class $F(\obj (\cu))^{\perp_{\du}}$ is closed with respect to $\du$-coproducts.

VI. Assume that $t$ is generated by $\cp$; denote by $\cu_{\cp}$ is the localizing subcategory of $\cu$ generated by $\cp$. Then there exists a $t$-structure  $t_{\cp}$ on $\cu_{\cp}$  that is generated by $\cp$ (in this category), and we have $\hrt_{\cp}=\hrt$.

\end{pr}
\begin{proof}
I.1. Obvious from Proposition \ref{prtst}(\ref{itperp}). 

2. We should prove that the object $H_0^t(\prod_{\cu}A_i)$ is the product of $A_i$ in $\hrt$. Now,  
 $\prod_{\cu}A_i\in \cu_{t\le 0}$ by the previous assertion, and 
 it remains to note that  $H_0^t(\prod_{\cu}A_i)\cong 
L_t(\prod_{\cu}A_i)$ according to Proposition \ref{prtst}(\ref{itho}), and apply the adjunction provided by Proposition \ref{prtst}(\ref{itcan}). 

3. Immediate from Proposition 1.2.1 
 of  \cite{neebook}.

II. Obvious.

III. Easy and well-known; see Proposition 3.4(1,2) of \cite{bvt}. 

IV. This is Theorem A.1 of \cite{talosa}.

V. All these statements except the last one are well-known as well; see Proposition 8.4.1, Theorem 8.3.3, 
  Proposition 8.4.6, and Theorem 8.4.4  of \cite{neebook}. 
   Furthermore, Lemma \ref{lper}(\ref{il7},\ref{il4},\ref{il4c}) below gives more detail on these matters.
 

The "moreover" part 
 assertion V.2 is an immediate consequence of  (the rather standard and easy) 
Proposition 3.4(5) of \cite{bvt}.

VI. The existence of $t_{\cp}$ is provided by assertion IV (applied to the category $\cu_{\cp}$). Next, the embedding $i:\cu_{\cp}\to\cu$ respects coproducts; since  $\cu_{\cp}$ is compactly generated, assertion V.2 implies that $i$ possesses a right adjoint. Hence Proposition \ref{prtst}(\ref{itdeg}) implies that $\hrt_{\cp}=\hrt$ indeed.
\end{proof}

\begin{rema}\label{rsmdual}
Certainly, the obvious categorical dual of part I of our proposition (that concerns $t$-structures on smashing triangulated categories; cf. Remark \ref{rtst}(1)) is valid as well. 
  Yet the duals to parts I.1 and I.2  will not be applied in the current paper. 

Recall also that both Proposition \ref{pcgt}(I) and its dual are essentially given by Proposition 3.2 of \cite{parrasao}.
\end{rema}

 We will also need the following key statement that is given by Corollary \ref{csymt} below.

\begin{lem}\label{lpcgt}
Assume that $t$ is a compactly generated $t$-structure. 

 Then the category $\hrt$ has an injective cogenerator. 
\end{lem}

Now we  establish a "significant part" of Theorem \ref{tgroth} (modulo Lemma \ref{lpcgt}). 

\begin{theo}\label{tab5}
Let $\cu$ be a smashing  triangulated category; 
let $t$ be a compactly generated $t$-structure on it.

Then the abelian category $\hrt$ is an AB5 one. \end{theo}
\begin{proof}
 Assume that $t$ is generated by a set $\cp\subset \obj\cu^{\alz}$. 
Then Proposition \ref{pcgt}(VI) allows us to replace $\cu$ by the corresponding subcategory $\cu_{\cp}$; thus  we can assume that $\cu$ is  compactly generated (as well). 
Hence $\cu$ is 
 cosmashing according to Proposition \ref{pcgt}(V.1). 

  Since $\hrt$ has an injective cogenerator according to Lemma \ref{lpcgt},  
    we  have all the ingredients needed for the criterion described in the introduction to \cite{posistov} (this if and only if statement is also 
 the categorical dual to Theorem 3.3 of ibid.).

 Consequently, for any family of $M_i\in \obj \hrt$ (that is indexed by a set $X$) it suffices to verify that the canonical morphism $\coprod_{\hrt} M_i\to \prod_{\hrt} M_i$ is monomorphic (actually, it suffices to take $M_i$ to be equal to a single object of $\hrt$ here; see condition (ii) in loc. cit.). 
  According to parts (II and)  I.2 and III.2 of   Proposition \ref{pcgt} we can present this morphism as $H_0^t(f)$, where $f$ is the canonical morphism $\coprod_{\cu} M_i\to \prod_{\cu} M_i$. Hence Proposition \ref{pcgt}(I.1, II) allows us to apply Proposition \ref{prtst}(\ref{itcone}) and to pass to checking that 
 $\co(f)\in \cu_{t\le 0}$. 

Now, 
  we can certainly replace the set $\cp$ by $\cup_{i\ge 0}\cp[i]$ (see Definition \ref{dtstro}(4)); we obtain $\cu_{t\le 0}=(\cp[1])\perpp$.	
 Next, for any $P\in \cp$ we have $\cu(P,  \coprod_{\cu} M_i)=\bigoplus \cu(P,   M_i)$ and (clearly) $\cu(P,  \prod_{\cu} M_i)=\prod \cu(P,   M_i)$. Hence the long exact sequence $$\dots\to \cu(P[1],  \prod_{\cu} M_i) \to \cu(P[1], \co(f))\to  \cu(P,  \coprod_{\cu} M_i)\to \cu(P,  \prod_{\cu} M_i)\to \dots $$ yields that $P[1]\perp \co(f)$. As we have just explained, this allows us to conclude the proof. 
\end{proof}

\subsection{Extensions of homological functors and generators for Ht}\label{scoext} 
 
Let us discuss the properties of certain extensions of homological functors from triangulated subcategories of compact objects.  
    Our construction is easily seen to be the standard pointwise construction of the corresponding left Kan extensions; yet we will not exploit this point of view below.

\begin{pr}\label{pkrause}

Let $\cuz$ be an essentially small triangulated subcategory of a smashing triangulated category $\cu$; let $H_0:\cu_0\to \au$ be a homological functor, where $\au$ is an AB5 abelian category. 
For any $M\in \obj \cu$ we fix a resolution
\begin{equation}\label{ekrause}
 \coprod_{i\in I}
H_{C_M^i}\to \coprod_{j\in J} 
H_{C_M^j}\to H_M\to 0,\end{equation} where  we use the notation $H_M$ for the restriction of the functor $\cu(-,M)$ to $\cuz$; the existence of a resolution of this sort is easy and demonstrated in the proof Lemma 2.2 of \cite{krause}.

Then for the association $H:M\mapsto \cok(\coprod H_0(C_M^i)\to  \coprod H_0(C_M^j))$  the following statements are valid.

 \begin{enumerate}
\item\label{ikr1} $H$ is a homological functor $\cu\to \au$. 

\item\label{ikrchar} For any $\adfu(\cuz\opp,\ab)$-resolution $ \coprod H_{C'{}_M^{i}}\to \coprod H_{C'{}_M^{j}} \to H_M\to 0$  of $H_M$,  where $C'{}_M^{i}$ and $C'{}_M^{j}$ 
 are  some objects of $\cuz$, the object $\cok(\coprod H_0(C'{}_M^{i})\to  \coprod H_0(C'{}_M^{j}))$ is canonically isomorphic to $H(M)$. 

In particular, the restriction of $H$ to $\cuz$ is canonically isomorphic to $H_0$. 

\item\label{ikradj} Let $\eu$ be a full smashing  triangulated subcategory of $\cu$ that contains $\cuz$ 
and  assume that there exists a right adjoint $i^*$ to the embedding $i:\eu\to \cu$. 
Then we have $H\cong H^{\eu}\circ i^*$, here the functor $H^{\eu}:\eu\to \au$ is defined on $\eu$ using the same construction as the one used for the definition of $\cu$.

\item\label{ikr8}  Assume that all objects of $\cuz$ are compact.
Then  $H$ 
is determined (up to a canonical isomorphism) by the following conditions: it respects coproducts, and its restriction to $\cuz$ equals $H_0$, and it kills $\cuz^\perp$.
\end{enumerate}
\end{pr}
\begin{proof}
\ref{ikr1}. Immediate from \cite[Lemma 2.2]{krause} (see also Proposition 2.3 of ibid.).

\ref{ikrchar}--\ref{ikradj}.
The proofs are straightforward (and very easy). 

\ref{ikr8}. 
In the case where $\cuz$ generates $\cu$ as its own localizing category the assertion is given by Proposition 2.3 of \cite{krause}. Now, in the general case the embedding of the  localizing category generated by $\cuz$ into $\cu$ possesses a
 right adjoint $i^*$ that respects coproducts according to Proposition \ref{pcgt}(V.2).   
Hence the general case of the assertion reduces to loc. cit. as well if we apply assertion \ref{ikradj}.
\end{proof}

Now we can ("almost") finish the proof of Theorem \ref{tgroth}.

\begin{coro}\label{cgroth}
Let $\cu$ be a smashing  triangulated category; let $t$ be a $t$-structure on it that is (compactly) generated by a set $\cp\subset \obj \cu^{\alz}$. 

Then the category $\hrt$ is Grothendieck abelian. Moreover, the 
 category $\cuz=\lan \cp \ra$ (see \S\ref{snotata}) is essentially small,  and the zeroth $t$-homology of 
 its objects give generators for $\hrt$.\footnote{Recall that a class $\cq\subset \obj \hrt$ is said to generate $\hrt$ 
 whenever for any non-zero $\hrt$-morphism $h$  there exists $Q\in \cq$ such that the homomorphism 
 $\hrt(Q,h)$ is non-zero as well. Since $\hrt$ is is closed with respect to small coproducts, if $\cq$ is essentially small then this condition is fulfilled if and only if any object of $\hrt$ is a quotient of a coproduct of elements of $\cp$.} 
 \end{coro}
\begin{proof}
Since $\hrt$ is an AB5 abelian category according to Theorem \ref{tab5}, it suffices to verify the second part of the statement.

Next, the category $\cuz$ is essentially small by Lemma 3.2.4 of \cite{neebook}. Hence Proposition \ref{pkrause}(\ref{ikr8}) implies that the functor $H^t:\cu\to \hrt$ is 
the corresponding (left Kan) extension of its restriction to the subcategory $\cuz$. 
 Hence for any $M\in \obj \cu$ and a  family $C_M^j\in \obj \cuz$ as in (\ref{ekrause}) we obtain that $H^t(M)$ is an $\hrt$-quotient of  $\coprod H^t(C_M^j)$. Thus the class $H^t(\cuz)$ generates $\hrt$ indeed. 
\end{proof}

\begin{rema}\label{rstalk} 1. In 
 \S5.4--5.5 of \cite{bpure} the author  studied the 
 category $\hrt$ under the assumption that there exists a smashing category $\du$ that contains $\cuz\opp$ as a full subcategory of compact objects. This 
 extra condition allowed to establish (in Theorem 5.4.2 of ibid.) the existence of an exact conservative functor $\mathcal{S}:\hrt\to \ab$ that respects coproducts (this functor was constructed as the coproduct of so-called stalk functors; see Remark 5.5.4(1) of ibid. for the motivation for choosing the last term). Now, this additional assumption appears to be rather harmless (since it is fulfilled at least whenever $\cu$ "has a model"; see Corollary 5.5.3 of ibid.), whereas the existence of a functor $\mathcal{S}$ of this sort is not automatic for Grothendieck abelian categories. 

2.  The author suspects that $\hrt$ possesses a much smaller class of generators; see Remarks 5.4.3(2) and 5.1.4(I.2) of ibid.
\end{rema}

\section{On 
 (perfectly generated) weight structures and adjacent $t$-structures}\label{spgws}

In this section we define the so-called perfectly generated weight structures. We also  use them to prove that hearts of compactly generated $t$-structures possess injective cogenerators, thus finishing the proof of Theorem \ref{tgroth}.

In \S\ref{scoulim} we recall the notion of a countable  homotopy colimit of a chain of morphisms in a smashing triangulated category. We also study the properties of colimits of this sort; some of them appear to be new (though rather technical). Probably, most of this section can be skipped at the first reading.

In \S\ref{sbws} we recall some basics on  weight structures; this notion  is central for the current paper.

In \S\ref{sspgws} we recall the notion of perfectness for classes of objects, and prove that any perfect set generates a 
 weight structure. We also establish 
  some properties of weight structures obtained this way.

In \S\ref{sadjts} we consider our main 
 example of perfect sets: we prove that Brown-Comenetz duals of  the elements of any set $\cp\subset \obj \cu^{\alz}$ 
  give a perfect set in the category $\cu\opp$. 
	 The corresponding weight structure on $\cu$ is right adjacent to the $t$-structure $t$ generated by $\cp$; 
	 this enables us to prove that the category $\hrt$ has an injective cogenerator.

\subsection{On homotopy colimits in triangulated categories}
\label{scoulim}

We recall the basics of the theory  of countable (filtered)
homotopy colimits in triangulated categories (as introduced in \cite{bokne}; some more detail can be found in \cite{neebook}). 
We will consider colimits of this sort only in triangulated categories that are {\it countably smashing}, i.e., closed with respect to countable coproducts (moreover, for the purposes of the current paper only smashing categories are actual); 
so usually we will not mention this (important!) restriction explicitly. 

\begin{defi}\label{dcoulim}
For a sequence of objects $Y_i$ 
 of $\cu$ for $i\ge 0$ and maps $f_i:Y_{i}\to Y_{i+1}$  
we consider the morphism $a:\oplus \id_{Y_i}\bigoplus \oplus (-f_i): D\to D$ (we can define it since its $i$-th component  can be  factored  into a composition $Y_i\to Y_i\bigoplus Y_{i+1}\to D$).  Denote  a cone of $a$ by $Y$. We will write $Y=\hinli Y_i$ and call $Y$ a {\it homotopy colimit} of $Y_i$ (we will not consider any other homotopy colimits in this paper).

 Moreover, $\co(f_i)$ will be denoted by $Z_{i+1}$, and we set $Z_0=Y_0$. 
\end{defi}

\begin{rema}\label{rcoulim}
1. Note that these homotopy colimits are not really canonical and functorial in $Y_i$ since the choice of a cone is not canonical. They are only defined up to non-canonical isomorphisms; still this is satisfactory for our purposes.

2. 
 The  definition of $Y$ gives a canonical morphism $D\to Y$; respectively, we also have canonical morphisms $Y_i\to Y$.

3. By Lemma 1.7.1 of \cite{neebook},  a homotopy colimit of $Y_{i_j}$ is the same (up to an isomorphism) for any subsequence of $Y_i$. 
In particular, we can discard any (finite) number of first terms in $(Y_i)$.

4. Most of  our 
 difficulties with (these) homotopy colimits in a triangulated category $\cu$ 
 are caused by the fact that they are not true $\cu$-colimits; consequently, we have to make much effort to control the difference in properties. However, the reader that is willing to ignore these technical 
 problems can easily note that the central ideas for the arguments that concern   homotopy colimits in this paper are rather transparent, whereas  the details for Lemmas \ref{lcoulim}(\ref{ihc5}) and \ref{ldualt} and their applications may be difficult to understand at the first reading.  \end{rema}

Let us now recall a few more properties of this notion. 

\begin{lem}\label{lcoulim} 
Assume that $Y=\hinli Y_i$ (in $\cu$); denote by $c_i$ the canonical morphisms $Y_i\to Y$ mentioned in Remark \ref{rcoulim}(2), and let $M$ be an object of $\cu$.
For  an abelian category $\au$ we assume that $H'$ (resp. $H$) is a 
 (co)homological functor from $\cu$ into $\au$.

Then the following statements are valid.

\begin{enumerate}
\item\label{ihc1}
If $f_i=\id_M$ 
 for all $i\ge 0$ then $c_i\cong \id_M$ as well; respectively, $Y\cong M$.

\item\label{ihc3} If the object $\inli H'(Y_i)$ (resp. $\prli H(Y_i)$) exists in $\au$ then the morphisms $H'(c_i)$ (resp. $H(c_i)$) induce a canonical morphism $\inli H'(Y_i)\to H'(Y)$ (resp. $H(Y)\to \prli H(Y_i)$).

\item\label{ihc4} Assume that $H$ is a cp functor. Then the aforementioned morphism $H(Y)\to \prli H(Y_i)$ is epimorphic. Moreover, it is an isomorphism whenever $\au$ is an AB4* category and all the morphisms $H(f_i[1])$ are 
epimorphic for 
 $i\gg 0$.

 \item\label{ihc5} Assume that $H'$ is a wcc functor. Then the aforementioned morphism $ \inli H'(Y_i)\to H'(Y) $ is monomorphic. 

This morphism is also an isomorphism  if either 

 (i) 
for  $i\gg 0$ there exist objects $A$ and $A_i\in \obj \au$,  along with compatible isomorphisms $H'(Y_i[1])\cong A_i\bigoplus A$ and $H(f_i[1])\cong (0:A_i\to A_{i+1}) \bigoplus \id_A$; 

or (ii) $\au$ is an AB5 category.

Furthermore, in case (i) we have $ \inli H'(Y_i[1])\cong A$. 

\item\label{ihc6} Assume that for each $i\in \z$ we are given  morphisms $m_i:Y_i\to M$ such that $m_i=m_{i+1}\circ f_i$ for all $i\ge 0$;  
we will call any preimage $m$ of the system $(m_i)$ with respect to the (surjective) homomorphism $\cu(Y,M)\to \prli \cu(Y_i,M)$ given by assertion   \ref{ihc4} (applied for $H=\cu(-,M)$) a morphism {\it compatible} with $(m_i)$. 

Then for any wcc functor 
 $H':\cu \to \au$ the restriction of $H'(b)$ to $\inli H'(Y_i)$ (see assertion   \ref{ihc5}) is given by $\inli H'(m_i)$.
\end{enumerate}
\end{lem}
\begin{proof}
\ref{ihc1}. This is Lemma 1.6.6 of \cite{neebook}.

\ref{ihc3}. It obviously suffices to verify the homological part of the assertion, since the cohomological one is its dual. Thus we should prove that for any $i\ge 0$ we have $H'(c_i)=H'(c_{i+1})\circ H'(f_i)$. Since $H'$ is homological, for this purpose it suffices to recall that the morphism  $Y_i\to D$ induced by $a$ (see Definition \ref{dcoulim}) equals $\id_{Y_i}\bigoplus (-f_i)$. 

\ref{ihc4}. We  have a long exact sequence
  $$\dots\to H(D[1])\stackrel{H(a[1])}{\longrightarrow} H(D[1])\to  H(Y)\to H(D)\stackrel{H(a)}{\longrightarrow} H(D)\to\dots.$$

	 Since $H(D)\cong \prod H(Y_i)$, the kernel of $H(a)$ equals $\prli H(Y_i)$ (and this inverse limit exists in $\au$), and we obtain the first part of the assertion. 
	
	Next, Remark A.3.6 of \cite{neebook} yields that the cokernel of $H(a[1])$ equals the $1$-limit of the objects $H(Y_i[1])$.
By Remark \ref{rcoulim}(3)  we can assume that the homomorphisms $f[1]^*$ are surjective for all $i$. Hence the statement is given by Lemma A.3.9  of ibid.

\ref{ihc5}. Similarly to previous proof 
 we consider the long exact sequence
$$\dots\to H'(D)\stackrel{H'(a)}{\longrightarrow} H'(D)\to H'(Y)\to  H'(D[1]) \stackrel{H'(a[1])}{\longrightarrow} H'(D[1])  \to\dots.$$

Since $H'(D)\cong \coprod H'(Y_i)$, it easily follows that the cokernel of $H'(a)$ is $\inli H'(Y_i)$; this gives the first part of the assertion.

To prove its second part  we should verify that $H'(a[1])$ is monomorphic (if either of the two additional assumptions is fulfilled).
We will write  $B_i$ and $g_i$ for $H'(Y_i[1])$ and $H'(f_i[1])$, respectively, whereas the morphism $H'(a[1])$ (that  clearly can be expressed in terms of $\id_{B_i}$ and $g_i$) will be denoted by $h$.

If (i) is valid then Remark \ref{rcoulim}(3) enables us to assume that $B_i\cong A\bigoplus A_i$ and $g_i\cong \id_A\bigoplus 0$ for all $i\ge 0$. Moreover, the additivity of the object $\ke(h)$ with respect to direct sums of $(B_i,g_i)$ reduces its calculation to the following two cases: (1) $(A=0;\ g_i=0)$ and (2) $(A_i=0;\ g_i\cong \id_A)$.
In 
 case (1) $h$ is isomorphic to $\id_{\coprod B_i}$; hence it is monomorphic. In 
 case (2) $h$ is monomorphic as well since 
  the morphism matrix $$\begin{pmatrix}\id_A &\id_A &\id_A &\dots \\
0 & \id_A &\id_A &\dots\\
0 & 0 &\id_A &\dots\\
0 & 0 &0 &\dots\\
\dots & \dots &\dots &\dots \end{pmatrix}$$ gives the inverse morphism (cf. the proof of \cite[Lemma 1.6.6]{neebook}). 

Moreover, the additivity of direct limits in abelian categories implies that  $\inli(H'(Y_i[1])\cong A\bigoplus A'$, where $A'$ is the direct limit of $A_0\stackrel{0}{\to} A_1\stackrel{0}{\to}A_2\stackrel{0}{\to}\dots$; clearly, $A'=0$.

To prove version (ii) of the assertion note that  the composition of $H'(a[1])$ with the obvious monomorphism $\coprod_{i\le j} H'(Y_i[1])\to \coprod_{i\ge 0} H'(Y_i[1])$ is easily seen to be monomorphic for  each $j\ge 0$. If  $\au$ is an AB5 category then  it follows that the morphism  $H'(a[1])$ is monomorphic itself.

\ref{ihc6}. It obviously suffices to note that the composition $Y_i\to Y\to M$ equals $m_i$ and apply $H'$ to this commutative diagram for all $i\ge 0$.
\end{proof}

We will also need the following definitions.\footnote{The terminology we introduce is new; yet 
big hulls were essentially considered 
in (Theorem 3.7 of) \cite{postov}.} 

\begin{defi}\label{dses}
 1. A class $\cpt\subset \obj \cu$ will be called {\it strongly extension-closed} if it 
 contains $0$ and for any $f_i:Y_{i}\to Y_{i+1}$ such that $Y_0\in \cpt$ and $\co(f_i)\in \cpt$ for all $i\ge 0$ we have $\hinli_{i\ge 0} Y_i\in \cpt$ (i.e. $\cpt$ contains all possible cones of the corresponding distinguished triangle; note that these are isomorphic).

2. The smallest strongly extension-closed retraction-closed class of objects of $\cu$ that  contains a class $\cp\subset \obj\cu$ and is closed with respect to arbitrary 
 $\cu$-coproducts will be called the {\it strong extension-closure} of $\cp$.

3. We will write $\cocp$ either for the closure of $\cp$ with respect to $\cu$-coproducts 
 or  for the full subcategory of $\cu$ formed by these objects. 

Moreover, we will call  the class of the objects of $\cu$ that may be presented as homotopy limits of  $Y_i$ with $Y_0$ and $\co(f_i)\in \cocp$, 
the {\it naive big hull} of $\cp$.  
The class of all retracts of the elements of 
 this naive big hull  will be called the {\it big hull} of $\cp$.
\end{defi}

Now we prove a few simple properties of these notions. 

\begin{lem}\label{lbes} 
Let $\cp$ be a class of objects of $\cu$; denote its strong extension-closure by $\cpt$.

\begin{enumerate}
\item\label{iseses}
Then $\cpt$ is extension-closed in $\cu$; it contains the big hull of $\cp$.
	 
\item\label{isesperp}
Let $H$ be a cp functor (see Definition \ref{dsmash}(\ref{idcc}))  from $\cu$ into a AB4*-category $\au$, and assume that  the restriction of $H$ to $\cp$ is zero. Then $H$  kills $\cpt$ as well.

In particular, if for some $D\subset \obj \cu$ we have $ \cp\perp D$ then $\cpt\perp D$ also.

\item\label{isescperp}
Let $H'$ be a cc functor from $\cu$ into a AB5-category. Then $H'$ kills $\cpt$ whenever it kills $\cp$.

 Thus if  $D\subset \obj \cu^{\alz}$ 
 and $D\perp \cp$ then $D\perp \cpt$ as well.

\item\label{izs}  {\it Zero classes} 
 of arbitrary families of cp and cc functors (into AB4* and AB5 categories, respectively) are strongly extension-closed (i.e. for any cp functors $H_i$ and cc functors $H'_i$ of this sort the classes $\{M\in \obj \cu:\ H_i(M)=0\ \forall i\}$ and $\{M\in \obj \cu:\ H'_i(M)=0\ \forall i\}$ are strongly extension-closed).
\end{enumerate}
\end{lem}
\begin{proof}
\ref{iseses}. For any  distinguished triangle $X\to Y\to Z$ for $X,Z\in \cpt$ the object $Y$ is the colimit of $X\stackrel{f}{\to} Y\stackrel{\id_Y}{\longrightarrow} Y \stackrel{\id_Y}{\longrightarrow} Y \stackrel{\id_Y}{\longrightarrow} Y\to \dots$; see Remark \ref{rcoulim}(3) and Lemma \ref{lcoulim}(\ref{ihc1}). Since
   a cone of $f$ is $Z$, whereas a cone of $\id_Y$ is $0$, 
	$\cpt$ is extension-closed indeed. It contains the big hull of $\cp$ by definition.

\ref{isesperp}. Since for any $d\in D$ the functor 
$H_d=\cu(-,d):\cu\to \ab$ converts arbitrary coproducts into products, it suffices to verify the first part of the statement.

Thus it suffices to verify that $H(Y)=0$ if $Y=\hinli Y_i$ and $H$ kills  cones of the connecting morphisms $f_i$.

Now, $H(Y_j)=\ns$ for any $j\ge 0$ (by obvious induction).  Next,  the  long exact sequence $$\dots  \to H(Y_{i+1}[1]) \stackrel{H(f_i[1])}{\longrightarrow} H(Y_i[1]) \to H(\co(f_i)) (=0) \to H(Y_{i+1}) \to H(Y_i)\to \dots  $$
gives the surjectivity of $H(f_i[1])$. Hence $H(Y)\cong \prli H(Y_i)=0$ according to Lemma \ref{lcoulim}(\ref{ihc4}).

\ref{isescperp}.  Once again, it suffices to verify the first part of the assertion. Similarly to the previous argument the result easily follows from  Lemma \ref{lcoulim}(\ref{ihc5}(ii)).

  \ref{izs}. 
	Immediate from the previous assertions. 
\end{proof}

We will also need the following lemma related to sequences of arrows.

\begin{lem}\label{ldualt}
Assume that  $h_i:M_{i}\to M_{i+1}$  for $i\ge 0$ is a sequence of $\cu$-morphisms, and for  $M\in \obj \cu$  we have connecting morphisms $g_i\in \cu(M,M_i)$ 
 such that $h_{i+1}\cong g_i\circ h_i$; take distinguished triangles $L_i\stackrel{b_i}{\longrightarrow}M \stackrel{g_i}{\longrightarrow}M_i\stackrel{f_i}{\longrightarrow} L_i[1]$ and $P_i\stackrel{a_i}{\longrightarrow} M_i\stackrel{b_i}{\longrightarrow} M_{i+1}\to P_i[1]$.  
Then there exists a system of morphisms $s_i:L_i\to L_{i+1}$ 
such that $b_{i+1}\cong s_i\circ b_i$. Moreover,  for $L=\hinli L_i$,  for any morphism $b:L\to M$ that is compatible with $(b_i)$ in the sense of Lemma \ref{lcoulim}(\ref{ihc6}), and any wcc functor $H:\cu\to \au$ the following statements are fulfilled.

1. $\co(s_i)\cong P_i$. 

2. 
 If $H(g_1)=0$ 
  then $H(b)$ is an epimorphism.

3. The morphism $H(b)$ is monomorphic whenever all the restrictions of $H(b_{i+1})$ to the images of $H(s_{i})$ are monomorphic  
and one of the following conditions are fulfilled: 
(i) $H(g_1)=0=H(g_1[1])$ and the restrictions of $H(b_{i+1}[1])$ to the images of $H(s_{i}[1])$ are monomorphic for $i\ge 1$ as well; 

 (ii) $\au$ is an AB5 category.  

4.  For any $i\ge 0$ the restriction of $H(b_{i+1})$ to the image of $H(s_{i})$ is monomorphic whenever $H(h_i[-1])=0$.

5. For any $i\ge 0$ if $H(a_i)$ is epimorphic then $H(h_i)=0$ and $H(g_{i+1})=0$. 

\end{lem}
\begin{proof} 
1. For all $i\ge 0$ we complete the commutative triangles $M\to M_{i}\to M_{i+1}$ to octahedral diagrams as follows:
\begin{equation}\label{eoct}
\xymatrix{M_{i+1}  \ar[dd]^{[1]} & & M \ar[ld]^{g_{i}} \ar[ll]^{g_{i+1}} \\ & M_i \ar[lu]^{h_i} \ar[rd]^{[1]}   & \\ P_i \ar[ru]^{a_i}
  \ar[rr]^{[1]} & & L_i\ar[uu]^{b_{i}}} 
\xymatrix{M_{i+1} \ar[rd]^{[1]}\ar[dd]^{[1]} & & M \ar[ll]^{g_{i+1}} \\ & L_{i+1} \ar[ru]^{b_{i+1}}\ar[ld]^{
} & \\ P_i\ar[rr]^{[1]} & & L_i\ar[lu]^{s_i}\ar[uu]^{b_{i}} }
\end{equation}
Consequently, we obtain the property 1 for this choice of $\{s_i\}$. Moreover, 
 we fix any morphism $b:L\to M$ that is compatible with $(b_i)$ (in the sense of Lemma \ref{lcoulim}(\ref{ihc6})).

2. 
Clearly, if $H(g_1)=0$ then $H(g_i)=0$ for all $i\ge 1$.

Next, the exact sequences 
\begin{equation}\label{eleso}
\dots\to H(L_i)\stackrel{H(b_i)}{\longrightarrow}H(M) \stackrel{H(g_i)}{\longrightarrow}H(M_i)\stackrel{H(f_i)}{\longrightarrow} H(L_i[1])\to \dots \end{equation}
 yield that $H(b_i)$ are epimorphic for $i\ge 1$. Lastly, 
 the first statement in Lemma \ref{lcoulim}(\ref{ihc5}) allows us to pass to the limit and conclude the proof. 

3. In version (i) (resp. (ii)) of our assertion we should prove that $H(b)$ is an isomorphism (resp. a monomorphism). Now, in case (ii) we have  $H(L)\cong \inli H(L_i)$ (see  Lemma \ref{lcoulim}(\ref{ihc5}(ii))). Since  $\inli H(L_i)\cong \inli \imm(H(s_{i}))$,   Lemma \ref{lcoulim}(\ref{ihc6}) gives the result in question. 

Next, in case (i) we clearly have $H(g_i[m])=0$ if $i\ge 1$ and $m$ equals $0$ or $1$; hence (\ref{eleso}) implies that 
 the corresponding $H(b_i[m])$ are epimorphic. Hence for any $i\ge 1$ and $m=0$ or $1$ we have $\imm H(s_i[m])\cong H(M[m])$, and the restriction of   $H(s_{i+1}[m])$ to $\imm H(s_i[m])$ is an isomorphism. Thus for any $i\ge 2$ the morphism $H(s_i[m])$ is isomorphic to $\id_{H(M)}\bigoplus 0:A_i^s\to A^s_{i+1}$ for certain $A^s_i\in \obj \au$. Hence applying Lemma \ref{lcoulim}(\ref{ihc5}(i)) (for $H'=H$ and also for $H'=H\circ [-1]$) 
 we obtain $H(L)\cong \inli H(L_i)\cong A$, and applying 
 Lemma \ref{lcoulim}(\ref{ihc6}) we conclude the proof.

4. The restriction of $H(b_{i+1})$ to the image of $H(s_{i})$ is monomorphic if and only if $H(s_{i})$ kills 
$\ke H(b_i)$. Next, $\ke H(b_i)=\imm (H(M_i[-1])\to H(L_i))$; thus we should check that the composition morphism $H(M_i[-1])\to H(L_{i+1})$ vanishes. Lastly, the octahedral axiom (also) says that   this composition can be factored through $H(h_i[-1])$, and we obtain the 
result in question.

5. The corresponding long exact sequence implies that $H(h_i)=0$ if and only if $H(a_i)$ is epimorphic. Next, if $H(h_i)=0$ then the morphism $H(g_i)$ is clearly zero as well. 
\end{proof}

\begin{rema}\label{r27}   Proposition 2.7 of \cite{postov} 
 gives a distinguished triangle $L\to M\to \hinli M_i\to L[1]$ whenever $\cu$ is a "stable derivator" triangulated category. We note that this additional assumption on $\cu$ is rather "harmless", and it can be used to simplify the proof of our lemma. However, it seems to be no way to avoid the assumptions similar to that in Lemma \ref{lcoulim}(\ref{ihc5}(i)) completely (for our purposes). 

2. Clearly, 
 instead of assuming that the morphisms $H(g_1)$ and $H(g_1[1])$ vanish 
 one can assume that  $H(g_i)$ and $H(g_i[1])$ vanish for 
 some $i>1$.\end{rema}

\subsection{Weight structures: basics}\label{sbws}
Let us recall the main definitions related to weight structures along with a few of their properties.

\begin{defi}\label{dwstr}

I. A couple of classes $\cu_{w\le 0},\cu_{w\ge 0}\subset\obj \cu$ 
will be said to define a weight structure $w$ on a triangulated category  $\cu$ if 
they  satisfy the following conditions.

(i) $\cu_{w\le 0}$ and $\cu_{w\ge 0}$ are 
retraction-closed in $\cu$ (i.e., contain all $\cu$-retracts of their objects).

(ii) {\bf Semi-invariance with respect to translations.}

$\cu_{w\le 0}\subset \cu_{w\le 0}[1]$, $\cu_{w\ge 0}[1]\subset
\cu_{w\ge 0}$.

(iii) {\bf Orthogonality.}

$\cu_{w\le 0}\perp \cu_{w\ge 0}[1]$.

(iv) {\bf Weight decompositions}.

 For any $M\in\obj \cu$ there
exists a distinguished triangle
$$LM\to M\to RM {\to} LM[1]$$
such that $LM\in \cu_{w\le 0} $ and $ RM\in \cu_{w\ge 0}[1]$.
\end{defi}

We will also need the following definitions.

\begin{defi}\label{dwso}
Let $i,j\in \z$; assume that a triangulated category $\cu$ is endowed with a weight structure $w$.

\begin{enumerate}
\item\label{idh}
The full subcategory $\hw$ of $ \cu$ whose objects are
$\cu_{w=0}=\cu_{w\ge 0}\cap \cu_{w\le 0}$ 
 is called the {\it heart} of 
$w$.

\item\label{id=i}
 $\cu_{w\ge i}$ (resp. $\cu_{w\le i}$, resp. $\cu_{w= i}$) will denote the class $\cu_{w\ge 0}[i]$ (resp. $\cu_{w\le 0}[i]$, resp. $\cu_{w= 0}[i]$).

\item\label{idwsmash} We will say that $w$ is {\it (co)smashing} if $\cu$ is (co)smashing and the class $\cu_{w\ge 0}$ (resp. $\cu_{w\le 0}$)  is (co)smashing in it.

\item\label{idrest}
Let $\du$ be a full triangulated subcategory of $\cu$.

We  say that $w$ {\it restricts} to $\du$ whenever the couple $(\cu_{w\le 0}\cap \obj \du,\ \cu_{w\ge 0}\cap \obj \du)$ is a weight structure on $\du$.

\item\label{idgenw} We will say that a class $\cp$ of objects of $\cu$ {\it generates}  $w$  whenever $\cu_{w\ge 0}=(\cup_{i>0}\cp[-i])\perpp$.

\item\label{idadj} 
 $w$ is said to be left (resp. right) adjacent to a $t$-structure $t$ on $\cu$ if $\cu_{w\ge 0}=\cu_{t\ge 0}$ (resp. $\cu_{w\le 0}=\cu_{t\le 0}$).

Moreover, if this is the case then we will also say that $t$ is right (resp. left) adjacent to $w$.

\item\label{idegw} $w$ is said to be {\it non-degenerate} if   $\cap_{i\in \z}\cu_{w\le i}=\cap_{i\in \z}\cu_{w\ge i}=\ns$  (cf. Definition \ref{dtstro}(3)).
\end{enumerate}
\end{defi}

\begin{rema}\label{rstws}

1. A  simple (and still  useful) example of a weight structure comes from the stupid filtration on the homotopy category of cohomological complexes $K(\bu)$ for an arbitrary additive  cohomological $\bu$ (it can also be restricted to bounded complexes; see Definition \ref{dwso}(\ref{idrest})). In this case $K(\bu)_{\wstu\le 0}$ (resp. $K(\bu)_{\wstu\ge 0}$) is the class of complexes that are
homotopy equivalent to complexes  concentrated in degrees $\ge 0$ (resp. $\le 0$); see Remark 1.2.3(1) of \cite{bonspkar} for more detail. 

2. A weight decomposition (of any $M\in \obj\cu$) is almost never canonical. 

Still for any $m\in \z$ the axiom (iv) gives the existence of a distinguished triangle \begin{equation}\label{ewd} w_{\le m}M\to M\to w_{\ge m+1}M\to (w_{\le m}M)[1] \end{equation}  with some $ w_{\ge m+1}M\in \cu_{w\ge m+1}$ and $ w_{\le m}M\in \cu_{w\le m}$; we will call it an {\it $m$-weight decomposition} of $M$.

 We will often use this notation below (even though $w_{\ge m+1}M$ and $ w_{\le m}M$ are not canonically determined by $M$); we will call any possible choice either of $w_{\ge m+1}M$ or of $ w_{\le m}M$ (for any $m\in \z$) a {\it weight truncation} of $M$. Moreover, when we will write arrows of the type $w_{\le m}M\to M$ or $M\to w_{\ge m+1}M$ we will always assume that they come from some $m$-weight decomposition of $M$.

3. In the current paper we use the ``homological convention'' for weight structures; 
it was originally introduced by J. Wildeshaus and  used in several preceding papers of the author.\footnote{Note also that this convention is compatible with the one used for weights of mixed complexes of \'etale sheaves in (\S5.1.5 of) \cite{bbd}; see Proposition 3.17 of \cite{brelmot}.} 
Note however that in \cite{bws} 
the ``cohomological convention'' was used. In the latter convention 
the roles of $\cu_{w\le 0}$ and $\cu_{w\ge 0}$ are essentially interchanged; being more precise, one uses the following notation:
  $\cu^{w\le 0}=\cu_{w\ge 0}$ and $\cu^{w\ge 0}=\cu_{w\le 0}$.
 
We also recall that D. Pauksztello has introduced weight structures independently 
 (in \cite{konk}); he called them co-t-structures. 
\end{rema}

Now we recall a collection of properties of weight structures.

\begin{pr}\label{pbw}
Let $\cu$ be a triangulated category endowed with a weight structure $w$. 
 Then the following statements are valid.

\begin{enumerate}
\item \label{idual}
The axiomatics of weight structures is self-dual, i.e., 
  the couple $(\cu_{w\ge 0}, \cu_{w\le 0})$ of objects of $\cu\opp$ gives   (the {\it opposite})  weight structure $w\opp$ on this category 
   (cf. Remark \ref{rtst}(1)).

\item\label{iort}
 $\cu_{w\ge 0}=(\cu_{w\le -1})^{\perp}$ and $\cu_{w\le 0}={}^{\perp} \cu_{w\ge 1}$.

\item\label{icoprod} $\cu_{w\le 0}$ is closed with respect to all coproducts that exist in $\cu$.

\item\label{iext} 
 $\cu_{w\le 0}$, $\cu_{w\ge 0}$, and $\cu_{w=0}$
are additive and extension-closed. 

\item\label{iwdext} For any distinguished triangle $M\to M'\to M''\to M[1]$  and any 
weight decompositions $LM\stackrel{a_{M}}{\longrightarrow} M\stackrel{n_{M}}{\longrightarrow} R_M\to LM[1]$ and $LM''\stackrel{a_{M''}}{\longrightarrow} M''\stackrel{n_{M''}}{\longrightarrow} R_M''\to LM''[1]$ there exists a commutative diagram 
$$\begin{CD}
LM @>{}>>LM'@>f>> LM''@>{}>>LM[1]\\
 @VV{a_M}V@VV{a_{M'}}V @VV{a_{M''}}V@VV{a_{M}[1]}V\\
M@>{}>>M'@>{}>>M''@>{}>>M[1]\\
 @VV{n_M}V@VV{n_{M'}}V @VV{n_{M''}}V@VV{n_{M}[1]}V\\
RM@>{}>>RM'@>{}>>RM''@>{}>>M[1]\end{CD}
$$
in $\cu$ whose rows are distinguished triangles and the second column is a weight decomposition (along with the first and the third one).

\item\label{igen} There exists at most one weight structure that is generated by a given class $\cp\subset\obj \cu$. Moreover, for a $t$-structure $t$ on $\cu$ there exists at most one weight structure $w_1$ (resp. $w_2$) that is left (resp. right) adjacent to $t$.

\item\label{iwdeg} Assume $\cu$ is a full strict triangulated subcategory of some $\du$; take $\eu$ to be the (triangulated; see Lemma \ref{lcloc})  subcategory of $\cu$-local objects in $\du$. 

Moreover, suppose that $w$ is generated by some $\cp\subset\obj \cu$ and there exists a functor right adjoint to the embedding $\cu\to \du$; take $\du_{w^{\du}\ge 0}$ to be the class of extensions of elements of 
 $\obj \eu$ by that of $\cu_{w\ge 0}$.   Then  $w^{\du}=(\cu_{w\le 0},\du_{w^{\du}\ge 0})$ is the weight structure generated by $\cp$ in   $\du$ (cf. the previous assertion). Consequently, $\hw^{\du}=\hw$.

\item\label{igenws} Assume that $w$ is smashing. Then the class $\cu_{w=0}$ is 
 smashing in $\cu$, and both $\cu$ and $\hw$ are {\it idempotent complete}, that is, every idempotent 
  endomorphism gives the projection of its domain onto its direct summand in 
  both of these categories.

Moreover, if $\cu$ is generated by a set of its objects as its own localizing subcategory then there exists 
 $P\in \cu_{w=0}$ such that any element of $\cu_{w=0}$ is a retract of a coproduct of 
 copies of $P$. 

\item\label{iexadj} Assume that the category 
 $\cu$ satisfies the Brown representability condition (see Definition \ref{dsmash}(\ref{idbrown})) and 
$w$ is  
smashing. Then there exists a 
$t$-structure $t$ on $\cu$ that is 
 right adjacent to $w$, and the heart $\hrt$ is equivalent (via the corresponding Yoneda-type functor) to the category of those 
 functors from $\hw\opp$ into $\ab$ that respect $\hw\opp$-products.  

\item\label{injadj} Assume that  $t$ is a  $t$-structure  on $\cu$ that is  left adjacent to $w$ and $\hw$ is idempotent complete (cf. assertion \ref{igenws}). Then the category $\hrt$ has enough injectives, and 
 the functor $H^t$ (see Proposition \ref{prtst}(\ref{itho}))  
 restricts to an equivalence of $\hw$ with  the subcategory of injective objects of $\hrt$. 
\end{enumerate}
\end{pr}
\begin{proof}

 Assertions \ref{idual}--\ref{iwdext} were proved in \cite{bws} (cf.  Remark 1.2.3(4) of \cite{bonspkar} and pay attention to Remark \ref{rstws}(3) above!). 

\ref{igen}. Immediate from  assertion \ref{iort}.

\ref{iwdeg}. 
  See Proposition 3.2(2,5) of  \cite{bvt} (cf. Proposition \ref{prtst}(\ref{itdeg})). 

\ref{igenws}. Assertion \ref{iort} implies that the class $\cu_{w\ge 0}$ is smashing in $\cu$. Recalling Definition \ref{dwso}(\ref{idwsmash}) we obtain that $\cu_{w=0}=\cu_{w\ge 0}\cap \cu_{w\le 0}$ is smashing as well.

Next, $\cu$   is smashing; hence  any idempotent endomorphism splits in  it by Remark 1.6.9 of \cite{neebook}. Since both $\cu_{w\le 0}$ and  $\cu_{w\ge 0}$ are retraction-closed in $\cu$, the same is true for $\cu_{w=0}$. Hence the category $\hw$ is idempotent complete as well.

Lastly, 
  Proposition 2.3.2(9) of \cite{bwcp} gives the existence of $P\in \cu_{w=0}$ such that any element of $\cu_{w=0}$ is a retract of a coproduct of 
 copies of $P$ immediately.

 \ref{iexadj}.  This is Theorem 3.2.3(I) of \cite{bvtr}.
 
 \ref{injadj}. 
  This  is a particular case  of \cite[Theorem 5.3.1(I.1)]{bvtr} (if we apply it to $\cu\opp$). 
\end{proof}

\begin{rema}
1. Moreover,  the proof of \cite[Theorem 5.3.1(I.1)]{bvtr} heavily relied on  Lemma 2(1) of \cite{zvon}, and this lemma implies  a significant part of  Proposition \ref{pbw}(\ref{injadj}) immediately.

2. Loc. cit. also says that the restriction of $H^t$ to $\hw$ is compatible in the obvious way with the Yoneda-type functor  $\hw\to \adfu(\hrt\opp,\ab)$, $M\mapsto \cu(-,M)$ (cf. Proposition \ref{pbw}(\ref{iexadj})).
\end{rema}

\subsection{On perfectly generated weight structures}\label{sspgws}

Now we recall the notion of a (weakly) perfect class of objects.

\begin{defi}\label{dpcl}
Let $\cp$ be a class of objects of $\cu$.

1. We will say that a $\cu$-morphism $h$ is {\it $\cp$-null} (resp. {\it $\cp$-epic})  whenever for all $M\in \cp$  we have $H^M(h)=0$ (resp. $H^M(h)$ is surjective), where 
 $H^M=\cu(M,-):\cu\to \ab$.

We will write  $\cpn$ for the class of all $\cp$-null morphisms.

2. Assume that $\cu$ is smashing. Then we will say that 
 $\cp$ is {\it (countably) perfect} if the class $\cpn$ 
 is closed with respect to (countable) $\cu$-coproducts. 
\end{defi}

\begin{rema}\label{requivdef} 1. Our definition of perfect classes essentially coincides with the one used in \cite{modoi}.

 Moreover, combining  (the obvious) Lemma \ref{lper}(\ref{il2}) below with Proposition \ref{pcgt}(I.3) (applied in the dual form) 
 one obtains the following: $\cp$ is (countably) perfect if and only if the coproduct of any (countable) family of $\cp$-epic morphisms is $\cp$-epic; hence $\cp$ is  countably perfect if and only if it fulfils  condition (G2)  in Definition 1 of 
 \cite{kraucoh}. 

2. Actually, the author does not know  any examples of countably perfect classes that are not perfect. Thus the reader can 
  assume that all the countably perfect classes mentioned below are just perfect.

3. The class of $\cp$-null morphisms is 
 not necessarily shift-stable  in contrast to the main examples of the paper  \cite{christ} where this notion was introduced.\end{rema}

Let us recall a few well-known facts related to perfect classes.

\begin{lem}\label{lper}
Let $\cu$ be a smashing triangulated category, $\cp\subset \obj \cu$, and $\cu'$ is the localizing subcategory of $\cu$ generated by $\cp$.

\begin{enumerate}

\item\label{il1} Then $\cp\perpp=\{N\in \obj \cu:\ \id_N\in \cpn\}$. Consequently, if  $\cp$ is (countably) perfect then  the class $\cp\perpp$ is smashing (resp. closed with respect to countable $\cu$-coproducts).

\item\label{il2}
In a $\cu$-distinguished triangle $M\stackrel{h}{\to}N \stackrel{f}{\to}Q\stackrel{g}{\to} M[1]$ the morphism $h$  is  $\cp$-epic if and only if $f$ is $\cp$-null.

\item\label{il3}
 The  class  $\ooo=(\cup_{i\in \z} \cp[i])^\perp$ equals 
  the one of $\cu'$-local objects; consequently,  the corresponding full subcategory $\du$ of $\cu$ is triangulated (cf. Lemma \ref{lcloc}). 
    
    Furthermore, $(\cup_{i\in \z} \cp[i])^ {\perp_{\cu'}}=\ns$.
    
\item\label{il4}    
Assume that $\cp$ is a (countably) perfect set. Then $\cp$ is (countably) perfect in $\cu'$; hence  the category $\cu'$ is {\it perfectly generated} in the sense of \cite[Definition 1]{kraucoh}. Consequently, $\cu'$ satisfies the Brown representability property 
     and  there exists an exact right adjoint $i^*$ to the embedding $i:\cu'\to \cu$ (cf.  Proposition \ref{pcgt}(V)).   
     
     Furthermore, if the class $\ooo$ (see assertion \ref{il3}) is zero then $\cu'=\cu$. 
    
  \item\label{il4c} Assume that   $\cu$ satisfies the Brown representability property. Then $\cu$ is cosmashing. 
   
   \item\label{il5}   
  The class $\cp[j]$ is (countably) perfect for any $j\in \z$ whenever $\cp$ is.
 
 \item\label{il6}   
 Let classes $\cp_i\subset \obj \cu$ for $i\in I$ be (countably) perfect.
Then $\cup_{i\in I}\cp_i$ is (countably) perfect as well.    
    
 \item\label{il7}       
   If all elements of $\cp$ are compact then $\cp$ is perfect. 
   \end{enumerate}
   \end{lem}
\begin{proof}
Assertions \ref{il1}--\ref{il2} and \ref{il5}--\ref{il7} are obvious.

\ref{il3}. 
 $\ooo$ is clearly stable with respect to shifts and extensions; hence the corresponding subcategory $\du$ is triangulated indeed. 
Next, the class $\perpp\obj\du$ is  extension-closed and  stable  with respect to shifts as well. Since it is also smashing, $\cu'\subset {}\perpp\ooo$; hence $\ooo=(\obj \cu')^\perp$ indeed. 
Lastly,  $(\obj \cu')^{\perp}\cap \obj \cu'=\ns$; thus  $(\cup_{i\in \z} \cp[i])^ {\perp_{\cu'}}=\ns$. 

\ref{il4}. Since $\cu'$ is a smashing subcategory of $\cu$, $\cp$ is (countably) perfect in $\cu'$ as well. Next, the previous assertion implies that the set $\cp$  satisfies condition (G1) of \cite[Definition 1]{kraucoh} in the category $\cu'$. As we have just noted, condition (G2) of loc. cit. is fulfilled for $\cp$ as well by assertion 2; hence  we can apply Theorem A of ibid. to obtain that $\cu'$ satisifes the  Brown representability property indeed. Thus  $i^*$ exists by Theorem 8.4.4 of \cite{neebook}. 
Lastly, if $\ooo=\ns$ then 
 $\cu=\cu'$ by the corollary in  \cite[\S1]{kraucoh}. 

\ref{il4c}. See Proposition 8.4.6 of 
\cite{neebook}. 
\end{proof}

Now we prove 
 the central theorem of the paper.

\begin{theo}\label{tpgws}
Let $\cp$ be a countably perfect 
set 
of  objects of a smashing category $\cu$; denote $\cup_{i\le 0}\cp[i]$ by $\cp'$.

1. Then   the couple $w=(L,R)$, where $L$ is 
 the big hull of $\cp'$ and $R=\cp'{}^\perp[-1]$,  is a countably smashing weight structure on $\cu$. 

Moreover, $L$ equals the strong extension-closure of $\cp'$. 

2.  Take $\cu'$ to be the localizing subcategory of $\cu$ generated by $\cp$ and  $\ooo=(\cup_{i\in \z}\cp[i])^{\perp_{\cu}}$. 

Then $\cp$ generates a  weight structure $w'$ in the category $\cu'$, 
    $\cu'_{w'\le 0}=L$, 
     $R$ equals the class of extensions of elements of 
            $\ooo$   by that of  
            $\cu'_{w'\ge 0}$, 
             and  $\hw'=\hw$.

3. Assume that $\cp$ is a perfect set. Then $w$ is smashing and  there exists 
 $P\in \cu_{w=0}$ such that any element of $\cu_{w=0}$ is a retract of a coproduct of 
 copies of $P$. 
\end{theo}
\begin{proof}
1. Denote the strong extension-closure of $\cp'$ by $L'$; clearly, $L'$ contains $L$.

Since $\cp'\perp R$, for any $N\in R$ the cp functor 
$H_N=\cu(-,N)$ 
 kills $L'$ according to Lemma \ref{lbes}(\ref{isescperp}). Hence $L'\perp R[1]$.

Next, $L$ is retraction-closed by definition, and obviously $R$ is  retraction-closed as well.

Now suppose that for any object $M$ of $\cu$ there exists a decomposition triangle 
\begin{equation}\label{edpg} LM\stackrel{b}{\to} M\stackrel{a}{\to} RM\stackrel{f}{\to} LM[1]\end{equation}
 with $LM\in L$ and $RM\in R[1]$. Then $(L,R)$ will give a weight structure just by definition, and one can apply Proposition \ref{pbw}(\ref{iort}) to obtain $L=L'$. 

So let us fix $M$ and construct a decomposition of the form (\ref{edpg}).
The idea is to apply Lemma \ref{ldualt}; 
 our argument is also related to the   proof of \cite[Theorem 4.5.2(I)]{bws} and to the construction of  crude cellular towers in \S I.3.2 of \cite{marg}.

 We construct a certain sequence of $M_k\in \obj \cu$ for $k\ge 0$ by induction in $k$ starting
from $M_0=M$. 
Assume  that $M_k$ (for some $k\ge 0$) is  constructed; then we take $P_k=\coprod_{(P,f):\,P\in \cp',f\in \cu(P,M_k)}P$; $M_{k+1}$ is a cone of the morphism $\coprod_{(P,f):\,P\in \cp',f\in \cu(P,M_k)}f:P_k\to M_k$.
Then compositions of the morphisms $h_k:M_{k}\to M_{k+1}$ given by this construction yields morphisms $g_i:M\to M_i$ for all $i\ge 0$.

We apply Lemma \ref{ldualt} 
 and obtain the existence of connecting morphisms $0=L_0\stackrel{s_0}{\longrightarrow}L_1\stackrel{s_1}{\longrightarrow}L_2\stackrel{s_2}{\longrightarrow}\dots$; we set $LM=\hinli L_i$. Moreover, we have a compatible system of morphisms $b_i:L_i\to M$ (cf. the formulation of that lemma) and we  choose  $b: LM\to M$ to be compatible with  $(b_k)$ (see Lemma \ref{lcoulim}(\ref{ihc6})). 
 We complete $b$ to a distinguished triangle $LM\stackrel{b}{\to} M\stackrel{a}{\to} RM\stackrel{f}{\to} LM[1]$;
it will be our candidate for a weight decomposition of $M$.

 Since $\co (s_i)\cong P_i$, the object $LM$ belongs to the naive big hull of $\cp'$  by the definition of this class.
It remains to prove that $RM\in R[1]$, i.e., that $\cp'\perp RM$. 
Let us apply an argument from the proof  of \cite[Theorem A]{kraucoh} (cf. also Remark \ref{rigid}(1) below).

We write $\cocpp$ for the full subcategory of $\cu$ formed by the closure of $\cp'$ with respect to coproducts.  
Following \cite{kraucoh} (see also \cite[Definition 5.1.3]{neebook} and \cite{auscoh}) 
we consider the full subcategory $\hatc\subset  
\adfu((\cocpp)^{op},\ab)$ 
 of {\it coherent functors}. 
We recall (see \cite{kraucoh}) that a  functor $H: (\cocpp)^{op}\to \ab$ 
is said to be coherent whenever there exists a $\adfu((\cocpp)^{op},\ab)$-short exact sequence $\cocpp(-,X)\to \cocpp(-,Y)\to H\to 0$, where $X$ and $Y$ are some objects of $\cocpp$ (note that this  is a projective resolution of $H$ in  $\adfu((\cocpp)^{op},\ab)$; see \cite[Lemma 5.1.2]{neebook}).

 According to \cite[Lemma 2]{kraucoh}, the category $\hatc$ is abelian; it has coproducts according to Lemma 1 of ibid. Since each morphism of (coherent) functors is compatible with some morphism of their (arbitrary) projective resolutions, a $\hatc$-morphism is zero (resp. 
 epimorphic) if and only if it 
 vanishes (resp. epimorphic) in $\adfu((\cocpp)^{op},\ab)$.

Next, take the Yoneda correspondence $\cu\to  \adfu((\cocpp)^{op},\ab)$  that maps $M\in \obj \cu$ into the restriction of $\cu(-,M)$ to $\cocp'$; it gives a  homological functor $H^{\cp'}:\cu\to \hatc$  (see Lemma 3 of ibid.). Now, $\cp'$ is countably perfect according to Lemma \ref{lper}(\ref{il5},\ref{il6}); thus Lemma 3 of ibid.  implies that $H^{\cp'}$ is a wcc functor. 
 $H^{\cp'}$ also respects arbitrary ${\cocpp}$-coproducts (very easy; see Lemma 1 of ibid.).
Lastly, our discussion of zero and surjective $\hatc$-morphisms clearly yields 
 for a  $\cu$-morphism $h$ that $H^{\cp'}(h)$ is zero (resp. epimorphic) 
   if and only if $h$ is a $\cocp'$-null morphism (resp. a $\cocp'$-epic one). The latter conditions is obviously fulfilled if and only if  $h$ is a $\cp'$-null morphism (resp. a $\cp'$-epic one).

Now we prove that $RM\in R[1]$ using the notation introduced in  Lemma \ref{ldualt}. 
As we have just proved, $RM\in R[1]$ whenever $H^{\cp'}(R)=0$. Hence the long exact sequence  $$
\to H^{\cp'}(L)\stackrel{H^{\cp'}(b)}{\longrightarrow}  H^{\cp'}(M) 
\to   H^{\cp'}(R)\to   H^{\cp'}(L[1])\stackrel{H^{\cp'}(b[1])}{\longrightarrow}  H^{\cp'}(M[1])\to
 $$
reduces the assertion to the epimorphness of $H^{\cp'}(b)$ along with the monomorphness of $H^{\cp'}(b[1])$.
Since $\cp'[-1]\subset \cp'$, the morphism $H^{\cp'}(a_i[j])$ is epimorphic (essentially by construction) for all $i,g\ge 0$.
Applying this observation in the case $i=j=0$ along with Lemma \ref{ldualt}(5,2) we obtain that $H^{\cp'}(b)$ is epimorphic. 

It remains to apply part 3(i) of the lemma to verify that $H^{\cp'}(b[1])$ is monomorphic; so we take $H=H^{\cp'}\circ [1]$. Thus we should check $H^{\cp'}(g_1[1])=0=H^{\cp'}(g_1[2])$ and also  that  $H^{\cp'}(h_i)=0=H^{\cp'}(h_i[1])$ for all $i\ge 0$ (see part 4 of the lemma). Thus combining part 5 of the lemma with the aforementioned surjectivity of $H^{\cp'}(a_i[j])$  we obtain that $w$ is a weight structure. 
 Lastly, Lemma \ref{lper}(\ref{il4}) implies that $w$ is  countably smashing.

2.  Lemma \ref{lper}(\ref{il4}) says that 
 $\cp$ is countably perfect in $\cu'$ and 
   there exists a right adjoint to the embedding $i:\cu'\to \cu$. 

Applying assertion 1 to the category of $\cu'$ we obtain that $\cp$ generates a weight structure $w'$ on it. Lastly, we apply Proposition \ref{pbw}(\ref{iwdeg}) to obtain that
 $\cu'_{w'\le 0}=L$, 
     $R$ equals the class of extensions of elements of 
            $\ooo$   by that of  $\cu'_{w'\ge 0}$, 
             and  $\hw'=\hw$ indeed.

3. Lemma \ref{lper}(\ref{il1}) implies that $w$ is smashing. 
By assertion 2, it remains  to verify the existence of $P$ in the case $\cu=\cu'$. In this case $P$ is given by Proposition \ref{pbw}(\ref{igenws}).
\end{proof}

\begin{rema}\label{rigid}  
1. The 
 author was inspired to apply coherent functors in this context by  \cite{salorio}; 
  yet 
the proof of Theorem 2.2 of ibid. (where coherent functors are applied to the construction of $t$-structures) appears to contain a gap.\footnote{An argument  even more closely related to our one was used in the proof of \cite[Lemma 2.2]{modoi}; yet the assumptions of that lemma appear to require  a correction.}
 It appears that applying 
  a similar argument to $\cp'=\cup_{i\ge 0}\cp$, where $\cp\subset \obj \cu$ is a general (countably) perfect set of objects, 
 one can (only) obtain a "weak $t$-structure" on $\cu$, i.e., for any $M\in \obj \cu$ there exists a distinguished triangle 
 $L\to M\to R\to L[1]$  such that $L$ belongs to the big hull of $\cp'$ and $R\in \cp'{}\perpp[1]$.

The author wonders whether this result can be improved, and also whether weak $t$-structures can be "useful".\footnote{Note however that {\it weak weight structures} (one replaces the orthogonality axiom in Definition \ref{dwstr}   by $\cu_{w\le 0}\perp\cu_{w\ge 2}$) were essentially considered in \cite{bsosn} (cf. Remark  2.1.2 of ibid.), in 
Theorem 3.1.3(2,3) of \cite{bsnew}, in \S3.6 of \cite{brelmot}, and in (Remark 6.3(4) of) \cite{bwvn}.}  

2. The case $\cp=\cp[1]$ of our theorem 
 is closely related to the proof of \cite[Theorem A]{kraucoh}. Respectively, we could have avoided citing loc. cit. in the proof of Lemma \ref{lper}(\ref{il4}).

3. We will say that a 
 weight structure is {\it perfectly generated} if it can be obtained by means of our theorem, i.e., if it is generated by a countably perfect set of objects.
 Note that Theorem \ref{twgws}(III.2) below states that any smashing weight structure on a  well generated triangulated category (see Definition \ref{dbecomp}(\ref{idpg})) is perfectly generated.

Clearly, instead of assuming that $\cp$ is a (countably perfect) set in our theorem it  suffices to assume that $\cp$ is 
 essentially small.

4. It is worth noting that in the setting of Theorem \ref{tpgws} 
 the category $\cu'$ is equivalent to the Verdier localization of $\cu$ by  $\du$, where $\obj \du=(\cup_{i\in \z}\cp[i])\perpp$ (see Lemma \ref{lwdeg}(\ref{ile3}) below).

5. In Theorem \ref{tsymt} (cf. also Theorem \ref{tpsymt})  
 we will study a family of examples for Theorem \ref{tpgws} that is  constructed using "symmetry"; this will yield some new results on $t$-structures. The idea to relate $t$-structures to symmetric sets and Brown-Comenetz duals comes from \cite{salorio} as well; however the author doubts that one can get a "simple description" of a $t$-structure obtained using arguments of this sort (cf. Corollary 2.5 of ibid.). 
\end{rema}


We also describe a certain "join" operation.

\begin{coro}\label{cwftw}
1. Assume that $\{\cp_i\}$ is a set of (countably) perfect  
sets of objects of $\cu$. Then the couple $w=(\cu_{w\le 0}, \cu_{w\ge 0})$ is a (countably) smashing weight structure on $\cu$, where $\cu_{w\le 0}$ is the big hull of $\cup_{j\ge 0,i} \cp_i[-j]$ and $\cu_{w\ge 0}=\cap_{j\ge 1,i} (\cp_i\perpp[-j])$. 

2. Assume that  $\{w_i\}$ is a set of perfectly generated weight structures (see Remark \ref{rigid}(3)) on $\cu$. 
Then the couple $w=(\cu_{w\le 0}, \cu_{w\ge 0})$ is a weight structure, where   $\cu_{w\le 0}$ is the big hull of $\cup_i \cu_{w_i\le 0}$ and  $\cu_{w\ge 0}=\cap_i \cu_{w_i\ge 0}$. Moreover, $w$ is perfectly generated; 
 it is  smashing whenever all $w_i$ are.

\end{coro}
\begin{proof}

1. $\cup_{i\in I} \cp_i$ is a  (countably) perfect set according to  Lemma \ref{lper}(\ref{il6}). 
 Hence $w$ is a (countably) smashing weight structure according to Theorem \ref{tpgws}(1,3). 

2. 
We choose countably perfect generating sets $\cp_i$ for all $w_i$. According to the previous assertion, the couple $(\cu_{w'\le 0}, \cu_{w'\ge 0})$ is a weight structure on $\cu$, where $\cu_{w'\le 0}$ is the big hull of $\cup_{j\ge 0,i} \cp_i[-j]$ and $\cu_{w'\ge 0}=\cap_{j\ge 1;i} (\cp_i\perpp[-j])$. 

Now we compare $w$ with $w'$. Since $\cp_i$ generate $w_i$,  $\cu_{w\ge 0}$ equals $\cu_{w'\ge 0}$. Next,  $\cu_{w\le 0}\perp \cu_{w\ge 0}[1]$ according to Lemma \ref{lbes}(\ref{isesperp}). Since $\cu_{w\le 0}$ contains $\cu_{w'\le 0}$, these classes are equal. 

Thus $w$ is a perfectly generated weight structure. It is smashing if all $w_i$ are; indeed,   $\cu_{w\ge 0}$ is smashing 
 since it is the intersection of smashing classes of objects of $\cu$.
\end{proof}

\begin{rema}\label{revenmorews} 
Part 2 of our corollary  gives a certain "join" operation on perfectly generated $t$-structures on $\cu$  (in particular, we obtain a monoid). 
 Note moreover that the join of any class  of  smashing  weight structures 
  is smashing as well if exists.
\end{rema}

\subsection{On the relation to adjacent (compactly generated) $t$-structures}\label{sadjts}

 To 
  construct weight structures adjacent to compactly generated $t$-structures we will use the following definitions.
 
 \begin{defi}\label{dsym}
Let $\cp$ and $\cp'$ be subclasses of $\obj \cu$, $P\in \obj \cu$.

\begin{enumerate}

\item\label{isym}
We will say that $\cp$ is {\it symmetric} to $\cp'$ if 
 $\cpn$ (see 
 Definition \ref{dpcl}(1)) coincides with  the class of {\it $\cp'$-conull} morphisms, that is, with the class of those $h\in \mo(\cu)$ such that $\cu(h,P')=0$ for every $P'\in \cp'$.

\item\label{ibcomo}
We will call an object of $\cu$ the {\it Brown-Comenetz dual} of $P$ and denote it by $\hat{P}$ if it represents 
 the functor $M\mapsto  \ab(\cu(M,-),\q/\z):\cu\opp\to \ab$.
 \end{enumerate}
 \end{defi}
 
 Now we prove some statements related to symmetry in the form sufficient to establish Lemma \ref{lpcgt}. Some stronger results can be found in 
  Theorem  \ref{tpsymt} below.

\begin{theo}\label{tsymt}
Assume that  $\cu$ is smashing and $\cp\in 
\obj \cu$.   

I.1. If $\cp$ is symmetric to 
  some class of objects of $\cu$ then  $\cp$ is perfect.
  
 2. If   $\hat{P}$ exists for any $P\in \cp$ then $\cp$ is symmetric to the class $\hat{\cp}=\{\hat{P}:\ P\in \cp\}$. 

 II. Assume in addition that $\cu$ satisfies the Brown representability property and all elements of $\cp$  
 are compact. 
 
 1. If  $P\in \cp$    then the Brown-Comenetz dual   $\hat{P}$ exists in it.\footnote{It appears that  this statement originates from \S2 of \cite{kraucoh}; cf. Theorem B of loc. cit. for an important application of this argument.}
 Consequently, $\cp$ is symmetric to  $\hat{\cp}$. 
 
  Moreover, the category $\cu^{op}$ is smashing and $\hat{\cp}$ perfect in it. 
 
 2. Suppose  furthermore  
   that $\cp$ is a set.  
  Denote by $w^{op}$ the weight structure 
  on 
   $\cu\opp$ that is generated by  $\hat{\cp}$; see 
    Theorem \ref{tpgws}. 
    Then the opposite weight structure $w$ on $\cu$ 
  (see Proposition \ref{pbw}(\ref{idual}))  is cosmashing and the class $\cu_{w=0}$ is cosmashing in $\cu$.
  Moreover,  $w$ is  right adjacent to the $t$-structure $t$ generated by $\cp$ on $\cu$  (as provided by Proposition \ref{pcgt}(IV)). 
  
   3.   
 Furthermore, the functor $H^t$ 
   gives an equivalence of $\hw$ with the subcategory of injective objects of $\hrt$. Moreover, 
  there  exists $I_w\in \cu_{w=0}$  such that any 
element of  $\cu_{w=0}$ is a retract of a product of copies of $I_w$, and $H^t(I_w)$ is an injective cogenerator of $\hrt$. 


\end{theo}
\begin{proof}
I.1. Obvious; note that every representable functor converts coproducts of morphisms into products of homomorphisms of abelian groups.

2. For a $\cu$-morphism $h$ and $P\in \cp$ the easy Proposition 4.3(5) of \cite{bvt} says that $h$ is $\{P\}$-null if and only if it is $\{\hat{P}\}$-conull.
Hence $\cp$ is symmetric to $\hat{\cp}$ by  Proposition 4.3(4) of loc. cit.

II.1. Since $\q/\z$ is an injective abelian group, the functor 
 $\hat{P}$ is cohomological. Moreover, it converts $\cu$-coproducts 
 into products of abelian groups; thus it is representable 
  by the definition of Brown representability. Hence $\cp$ is symmetric  to $\hat{\cp}$ by assertion I.2.

Next,  the category $\cu\opp$ is smashing by Lemma \ref{lper}(\ref{il4c}).  
   Hence $\hat{\cp}$ is perfect in  $\cu\opp$ 
    by  assertion I.1. 
 
2. 
 Theorem \ref{tpgws} implies that $w$ is cosmashing and the class $\cu_{w=0}$ is cosmashing in $\cu$.

 Next, Proposition 3.2(1,2) of \cite{bvt} allows us to apply Proposition 4.4(4) 
 of ibid. to our context; we conclude that $w$ is right adjacent to $t$ indeed. 

3. 
By Proposition \ref{pbw}(\ref{injadj},\ref{igenws}),   the restriction of  $H^t$  to $\hw$  gives an equivalence of 
 $\hw$ to  the subcategory of injective objects of $\hrt$, and $\hrt$ has enough injectives. 
 
 Now we apply Theorem \ref{tpgws}(3)   to obtain the  existence of $I_w$ such that any 
element of  $\cu_{w=0}$ is a $\cu$-retract of a product of copies of $I_w$. 
 Hence the object $H^t(I_w)$ is an injective cogenerator of $\hrt$ indeed.
\end{proof}

\begin{rema}\label{rvtt} 

1.  Let us prove that for any non-zero compact object $P$ of $\cu$ the object $\hat{P}$ is not compact in $\cu\opp$. Take  an infinite family of $C_i\in \obj \cu,\ i\in I$, such that  all the groups $H_i=\cu(P,C_i)$ are non-zero (one can just take $I=\z$ and all $C_i=P$) and choose non-zero $f_i:H_i\to \q/\z$. Since $\q/\z$ is an injective abelian group, the homomorphism $\sum f_i:\bigoplus H_i\to \q/\z$ can be factored as\\ $\bigoplus_I H_i\stackrel{i}{\to} \prod_I H_i\stackrel{f}{\to}  \q/\z$, where $i$ is the corresponding embedding. Then the morphism $C= \prod_I C_i\to \hat{P}$ corresponding to $f$ (note that $\cu(\prod C_i, \hat{P})\cong \ab(\cu(P,\prod C_i),\q/\z)\cong  \ab(\prod H_i, \q/\z)$) does not factor through the projection of $C$ onto any finite direct sum of $C_i$; hence $\hat{P}$ is not compact in $\cu\opp$ indeed.
  
 2. Moreover, note  that on $\cu=D(\q)$ (the derived category of $\q$-vector spaces)  there is a  canonical $t$-structure that is generated by the compact object $\q$ (placed in degree $0$); this $t$-structure is non-degenerate. Yet the corresponding weight structure $w\opp$ on $\cu\opp$ cannot be compactly generated since there are no non-zero compact objects in $\cu\opp$; cf. \S E.2 of \cite{neebook}.
  Consequently, one cannot apply Theorem 5 of \cite{paucomp} instead of Theorem \ref{tpgws} in the proof of Theorem \ref{tsymt}(II).

  More generally, the author suspects that  $w\opp$ is never compactly generated if $t$ is (compactly generated and) non-degenerate. Corollary E.1.3 of  \cite{neebook} gives some evidence for this conjecture.
  
\end{rema}

Now we are finally able to prove the following long-awaited for statement.

\begin{coro}\label{csymt}
If $t$ is a compactly generated $t$-structure then the category $\hrt$ possesses an injective cogenerator.
\end{coro}
\begin{proof} 
According to Proposition \ref{pcgt}(VI) (cf.  the proof of Theorem \ref{tab5}) we can assume that $\cu$ is 
compactly generated (by $\cp$). 
	 Then $\cu$ satisfies the Brown representability condition by Lemma \ref{lper}(\ref{il4}), and we can apply Theorem \ref{tsymt}(II.3) to obtain the result.
\end{proof}

\begin{rema}\label{r438}
1. In  the case where $t$ is non-degenerate one may apply some alternative arguments to study $\hrt$.  One can use the existence of the right adjacent weight structure $w$ provided by Theorem \ref{tsymt}(II.2)   similarly to the proof of \cite[Corollary 4.9]{humavit} (and using  Theorems 4.8 and 3.6 of ibid.) to 
 prove that $\hrt$ is Grothendieck abelian; see the proof of \cite[Corollary 4.3.9]{bpure} for more detail.

2. 
 $t$ is automatically left non-degenerate if (and only if) $\cu$ is (compactly) generated by the set $\cp$. Moreover, the localizing subcategory of $\cu$ generated by $\cp$ is certainly compactly generated, and the hearts of the corresponding $t$ and $w$ can be computed using this subcategory. We will justify this (somewhat vague) claim in \S\ref{sdeg} below, where 
  the assumptions of Theorem \ref{tsymt}(II.3) will be weakened. However, the $t$-structures and weight structures obtained via arguments of this sort are always degenerate; see  Proposition \ref{pcgt}(VI) and  Proposition \ref{pbw}(\ref{iwdeg}). 
 
3. It appears to be (much more) difficult to "control" the right non-degeneracy of a compactly generated $t$-structure. 
 Respectively, the author does not know how to 
 use the arguments mentioned in part 1 of this remark to establish Theorem \ref{tab5} in general. 
\end{rema}

\section{On torsion theories and well-generated weight structures}\label{swgs}

In this section we give a certain classification of compactly generated  torsion theories; those essentially generalize both weight structures and $t$-structures. 

So,  in \S\ref{stt} we recall some basics on torsion theories. 

In \S\ref{scghop} we prove that compactly generated torsion theories on a given smashing triangulated category $\cu$ are in a natural one-to-one correspondence with extension-closed retraction-closed essentially small classes of compact objects of $\cu$. We also discuss their relation of our results to the ones of \cite{postov} (that are essentially just a little less general than our ones) and to Theorem  A.9 of \cite{kellerw}. 

Lastly, in \S\ref{swgws} we study general smashing torsion theories; our results yield that that all smashing weight structures on well generated triangulated categories are perfectly generated (and more than that).

\subsection{Torsion theories: basic definitions and properties}\label{stt}

Let us recall  basic definitions for torsion theories.

\begin{defi}\label{dhop}
\begin{enumerate}
\item\label{ittd}
 A couple $s$ of classes $\lo,\ro\subset\obj \cu$ 
will be said to be a {\it torsion theory} (on $\cu$) if $\lo^{\perp}=\ro$,  $\lo={}^{\perp}\ro$, and 
for any $M\in\obj \cu$ there
exists a distinguished triangle
\begin{equation}\label{swd}
L_sM\stackrel{a_M}{\longrightarrow} M\stackrel{n_M}{\longrightarrow} R_sM
{\to} L_sM[1]\end{equation} 
such that $L_sM\in \lo $ and $ R_sM\in \ro$. We will call any triangle of this form an {\it $s$-decomposition} of $M$; $a_M$ will be called an {\it $s$-decomposition morphism}.

\item\label{ittg}
We will say (following \cite[Definition 3.1]{postov}) that $s$ is {\it generated by $\cp\subset \obj \cu$} if $\cp^\perp=\ro$. 

Moreover, if $\cp$ is a set of compact objects then we will say that $s$ is {\it compactly generated}.

\item\label{ittsm}
 $s$ is  said to be (co)smashing if $\cu$ is (co)smashing and $\ro$ is smashing (resp. $\lo$ is cosmashing) in it.

\item\label{ittr}
 If $\cu'$ is a full triangulated subcategory of $\cu$ then we say that $s$ {\it restricts} to it whenever $(\lo\cap \obj \cupr,\ro\cap \obj \cupr)$ is a torsion  theory on $\cu'$.
\end{enumerate}
\end{defi}

\begin{pr}\label{phop}
Let $s=(\lo,\ro)$ be a torsion theory on $\cu$. Then the following statements are valid.

\begin{enumerate}
\item\label{itpt} There is a 1-to-1 correspondence between those torsion theories such that $\lo\subset \lo[-1]$ and $t$-structures; it is given by sending $s=(\lo,\ro)$ into $(\cu_{t\le 0}=\ro[1],\ \cu_{t\ge 0}=\lo)$. 

We will say that $t$ is associated with $s$ and $s$ is associated with $t$ in this case.

\item\label{itpw} There is a 1-to-1 correspondence between those torsion theories such that $\lo\subset \lo[1]$ and weight structures; it is given by sending $s=(\lo,\ro)$ into $(\cu_{w\le 0}=\lo,\ \cu_{w\ge 0}=\ro[-1])$. 

 $w$ is said to be associated with $s$ and $s$ is associated with $w$ in this case; we will also say that $s$ is {\it weighted}.

\item\label{itpsm} If $s$ is associated with a $t$-structure $t$ (resp. a weight structure $w$) then $s$ is (co)smashing if and only if $t$ (resp. $w$ is).

\item\label{itpg} If $s$ is generated by a class $\cp\subset \obj\cu$ and a torsion theory $s'$ satisfies this property as well then $s=s'$ and $\cp\subset \lo$.

\item\label{itp1}
Both $\lo$ and $\ro$ are retraction-closed and extension-closed in $\cu$.

\item\label{itp4}
Assume that $\cu$ is (co)smashing.  Then  $s$ is  (co)smashing if and only if the coproduct (resp., product) of  any $s$-decompositions of $M_i\in \obj \cu$ gives an $s$-decomposition of $\coprod M_i$ (resp. of $\prod M_i$). 

\item\label{itp7}
$s$-decompositions are "weakly functorial" in the following sense:  any $\cu$-morphism $g:M\to M'$ can be completed to a morphism between any choices of $s$-decompositions of $M$ and $M'$, respectively. 

\item\label{itp7p}
 If $M\in \lo$ and $M$ is a retract of $M'\in \obj \cu$ then  is also a retract of any choice of $L_sM'$ (see  Remark \ref{rtt}(\ref{irtt1}) below).

\item\label{itp8}
A morphism $h\in \cu(M,N)$ is $\lo$-null (see Definition \ref{dpcl}(1)) if and only if factors through  an element of $\ro$.

\item\label{itp9}
For $L,R\subset \obj \cu$ assume that $L\perp R$ and that for any $M\in \obj \cu$ there exists a distinguished triangle $l\to M\to r\to l[1]$  for $l\in L$ and $r\in R$. Then $(\kar_{\cu}(L),\kar_{\cu}(R))$ is a  torsion theory on $\cu$; here for a class $D$ of objects of $\cu$ we use the notation $\kar_{\cu}(D)$ for the class of all $\cu$-retracts of elements of $D$. 

\item\label{itpcg} If $s$ is compactly generated then it is smashing.
\end{enumerate}
\end{pr}
\begin{proof}

These statements are mostly easy, and all of them except 
assertions \ref{itpsm}, \ref{itp7p}, and \ref{itpcg} were established in \cite{bvt} (see Propositions 3.2(1,2) and 2.4 of ibid.).

Next,  assertions \ref{itpsm} and \ref{itpcg} are obvious.

Lastly,  if $M$ is a retract of $M'$ then $\id_M$ factors through $M'$. Moreover, if $M\in \lo$ then the distinguished triangle $M\to M\to 0\to M[1]$ is an $s$-decomposition of $M$. Thus applying assertion \ref{itp7} twice we obtain a commutative diagram $$\begin{CD}
 M@>{g}>>L_sM'@>{h}>>M\\
@VV{\id_M}V@VV{a_{M'}}V@VV{\id_M}V \\
M@>{i}>>M'@>{p}>>M
\end{CD}$$ 
which yields that $h\circ g=\id_M$, i.e., that $M$ is a retract of $L_sM'$ indeed.
\end{proof}

\begin{rema}\label{rtt}
\begin{enumerate}
\item\label{irtt1}
 The object $M$ "rarely" determines its $s$-decomposition triangle (\ref{swd}) canonically; see Remark \ref{rstws}(2) along with Proposition \ref{phop}(\ref{itpw}). Yet we will often need some choices of its ingredients; so we will use the notation of (\ref{swd}).

\item\label{irtt2}
Our definition of torsion theory actually follows \cite[Definition 3.2]{postov}  and is somewhat different
from Definition 2.2 of \cite{iyayo}, from which our term comes from.  However, Proposition \ref{phop}(\ref{itp1},\ref{itp9})  easily implies that these two definitions are equivalent; see also Remark 2.5(1) of \cite{bvt} for some more detail.
\end{enumerate}
\end{rema}

\subsection{A classification of compactly generated torsion theories}\label{scghop}

\begin{theo}\label{tclass}
Assume that  
 $\cp\subset \obj \cu$ 
is a  set of compact objects.

Then the following statements are valid.
\begin{enumerate}
\item\label{iclass1}
The strong extension-closure $\lo$   of $\cp$ (see Definition \ref{dses}) and $\ro=\cp^\perp$ give a smashing torsion theory $s$ on $\cu$ (consequently, $s$ is the torsion theory generated by $\cp$). Moreover, $\lo$ equals the big hull of $\cp$, and for any $M\in \obj \cu$ there exists a choice of $L_sM$ (see   Remark \ref{rtt}(\ref{irtt1})) that belongs to the naive big hull of $\cp$.

\item\label{iclass2} The class of compact objects in $\lo$ equals the $\cu$-envelope of $\cp$ (see \S\ref{snotata}).

\item\label{iclassts} The correspondence sending a compactly generated torsion theory  $s=(\lo,\ro)$ for $\cu$ into $\lo\cap\obj \cu^{\alz}$ (i.e., we take compact objects in $\lo$; see Definition \ref{dsmash}(\ref{icomp})), 
 gives a one-to-one correspondence between the following classes: the class of  compactly generated torsion theories on $\cu$ and the class of essentially small 
retraction-closed extension-closed subclasses  of $\cu^{\alz}$.
\footnote{Actually, $\cu^{\alz}$ is essentially small itself in "reasonable" cases; in this case the essential smallness of  classes of objects in it is automatic.} 

\item\label{iclasst}  If the torsion theory  $s$ generated by $\cp$ is associated to a $t$-structure then the class $\lo$ ($=\cu_{t\ge 0}$)  equals the naive big hull of $\cp$.

\item\label{iclass5}  Let 
$H$ be a cp (resp. a cc) functor from $\cu$ into an AB4* (resp. AB5) category $\au$ whose restriction to $\cp$ is zero. Then $H$ kills all elements of $\lo$ as well. 
\end{enumerate} 
\end{theo}

\begin{proof}
\begin{enumerate}
\item If $s$ is a torsion theory indeed then it is smashing according to Proposition \ref{phop}(\ref{itpcg}).

Since $\cp\perp \ro$, for any $N\in \ro$ the cp functor 
$H_N=\cu(-,N)$ 
 kills $\lo$ according to Lemma \ref{lbes}(\ref{isescperp}). Hence $\lo\perp\ro$.\footnote{This statement was previously proved in \cite{postov} and our argument is just slightly different from the one of Pospisil and \v{S}\v{t}ov\'\i\v{c}ek; see Lemma 3.9 of  ibid.}
Since $\lo$ is retraction-closed by definition, Proposition \ref{phop}(\ref{itp9}) (along with Lemma \ref{lbes}(\ref{iseses})) reduces the assertion  to the existence for any  $M\in \obj\cu$ of  an $s$-decomposition 
 such that the corresponding $L_sM$   belongs to the naive big hull of $\cp$. Now we argue similarly to the proof of Theorem \ref{tpgws}(1); yet we apply "the easier version of" Lemma \ref{ldualt}(3). 

  We fix $M$ and construct a certain sequence of $M_k\in \obj \cu$ for $k\ge 0$ by induction in $k$ starting
from $M_0=M$. 
Assume  that $M_k$ (for some $k\ge 0$) is  constructed; then we take $P_k=\coprod_{(P,f):\,P\in \cp,f\in \cu(P,M_k)}P$; $M_{k+1}$ is a cone of the morphism $\coprod_{(P,f):\,P\in \cp,f\in \cu(P,M_k)}f:P_k\to M_k$.
Then compositions of the morphisms $h_k:M_{k}\to M_{k+1}$ given by this construction yields morphisms $g_i:M\to M_i$ for all $i\ge 0$.

We apply Lemma \ref{ldualt} 
 and obtain the existence of connecting morphisms $0=L_0\stackrel{s_0}{\longrightarrow}L_1\stackrel{s_1}{\longrightarrow}L_2\stackrel{s_2}{\longrightarrow}\dots$; we set $L=\hinli L_i$. Moreover, we have a compatible system of morphisms $b_i:L_i\to M$ (cf. the formulation of that lemma) and we  choose  $b: L\to M$ to be compatible with  $(b_k)$ in the sense of Lemma \ref{lcoulim}(\ref{ihc6}). 
 We complete $b$ to a distinguished triangle $L\stackrel{b}{\to} M\stackrel{a}{\to} R\stackrel{f}{\to} L[1]$;
it will be our candidate for an $s$-decomposition of $M$.

 Since $\co (s_i)\cong P_i$,   $L$ belongs to the naive big hull of $\cp$ by the definition of this hull.

It remains to prove that $R\in \ro$, i.e., that $\cp\perp R$. 
For an element $P $ of $\cp$ we should check that $\cu(P,R)=\ns$, i.e.,  for the functor $H^P=\cu(P,-)$ we should prove that $H^p(R)=0$.

The long exact sequence  $$\dots  \to \cu(P,L)\to  \cu(P,M)\to \cu(P,R)\to \cu(P, L[1])\to \cu(P,M[1])\to\dots $$ translates this into the following assertion: $H^P(b)$ is surjective and $H^P(b[1])$ is injective. 

Now, $H^P$ is clearly a wcc (and actually a cc) functor, and its target is an AB5 category. Since $H^P(a_i)$ is epimorphic by construction for all $i\ge 0$, Lemma \ref{ldualt}(5,2) implies the surjectivity of $H^P(b)$. 

Next 
we apply   Lemma \ref{ldualt}(4,3(ii)) for $H=H^P\circ [1]$ and obtain that to verify the injectivity of  $H^P(b[1])$ it remains to check that $H^P(h_i)=0$ for $i\ge 0$. 
Applying part 5 of the lemma we reduce the statement in question to the aforementioned surjectivity of  $H^P(a_i)$. 
  
\item Recall that the category $\cu^{\alz}$ (see Definition \ref{dsmash}(\ref{icomp})) is a full triangulated subcategory of $\cu$; moreover, this subcategory is obviously retraction-closed in $\cu$. Hence the class $\obj\cu^{\alz}\cap \lo$ contains the envelope of $\cp$.

 Next, the smallest strict triangulated  subcategory of $\cu$ 
 containing $\cp$ is essentially small by  Lemma 3.2.4 of \cite{neebook}. Hence $\lan \cp\ra_{\cu}$   is essentially small as well (cf. Proposition 3.2.5 of ibid.); thus  the envelope of $\cp$ also is. 
  Hence we should prove that 
 the class $\obj\cu^{\alz}\cap \lo$ equals $\cp$ whenever 
 $\cp$ is essentially small, retraction-closed, and extension-closed in $\cu$.  

Now, Corollary 3.11 of  \cite{bsnull} 
 (applied to the category $\cuz\opp$) gives the following remarkable statement: if $\cuz$ is a small triangulated category then a set $\cp_0$ of its objects is the zero class (see Lemma \ref{lbes}(\ref{izs})) of some "detecting" homological functor $H_0:\cuz\to \ab$ if and only if $\cp_0$ is extension-closed and retraction-closed in $\cuz$.
We take $\cuz$ to be  a small skeleton of the category $\lan \cp\ra$, $\cp_0=\obj \cuz\cap \cp$, and take  $H_0$ to be the corresponding "detector functor".  
Since all objects of $\cuz$ are compact in $\cu$,  the left Kan extension $H$ of $H_0$  to $\cu$ (as provided by Proposition \ref{pkrause}) is  a cc functor according to Proposition \ref{pkrause}(\ref{ikr8}).  Similarly to the proof of Theorem \ref{tsymt} we take the Brown-Comenetz dual functor $\hdu$ from  $\cu$ into $\ab$, $M\mapsto \ab(H(M),\q/\z)$.
This a cp functor from $\cu$ into $\ab$, 
 and its zero class 
  coincides with that of $H$. 

We take $\cu'$ to be localizing subcategory of $\cu$ generated by $\cp$; clearly, this subcategory contains $\lo$ (see the previous assertion). Moreover, Lemma 4.4.5 of \cite{neebook} implies that that the subcategory $\cu'{}^{\alz}$ essentially equals $\cuz$; hence the class  $\obj \cu'\cap \obj \cu^{\alz}$ 
 essentially equals $\obj \cuz$. 

Since $\cu'$ is generated by a set of compact objects as its own localizing subcategory, $\cu'$ satisfies the Brown representability condition according to Proposition \ref{pcgt}(V.1).  Thus the restriction $\hdu'$ of the functor $\hdu$  to $\cu'$ is $\cu'$-representable by some $I\in \obj\cu'$.  Since the zero class of $\hdu'$  contains $\cp$, we have $I\in \ro$. Hence for $M\in  \obj \cu'\cap \obj \cu^{\alz}$ we have $M\in \lo$ if and only if $M\in \cp$.

\item We take the envelope $\cp'$ of $\cp$. As we have just shown, $\cp'$ is essentially small, and since $\cp'\perpp=\cp\perpp$ (see Proposition \ref{phop}(\ref{itp1})),   the torsion theory $s$ given by 
assertion \ref{iclass1} is also generated by $\cp'$. Hence it suffices to note that $\lo\cap \cu^{\alz}=\cp'$ according to assertion \ref{iclass2}.

\item Recall from 
 assertion \ref{iclass1} 
that for any $M\in \obj \cu$ there exists a choice of $L_sM$ that belongs to the naive big hull of $\cp$. Now, if $s$ is associated to a $t$-structure and $M\in \lo=\cu_{t\ge 0}$ then we  have $L_sM=M$ according to Proposition \ref{prtst}(\ref{itcan}); this concludes the proof.

\ref{iclass5}. Immediate from Lemma \ref{lbes}(\ref{isesperp}) (resp. \ref{isescperp}).
\end{enumerate}
\end{proof}

\begin{rema}\label{rnewt}
1. Recall that Theorem 3.7 and  Corollary 3.8 of \cite{postov} give parts \ref{iclass1}--\ref{iclassts} of our theorem in the case where $\cu$ is  a "stable derivator" triangulated category $\cu$. 

Note however that the existence of a "detector object" $I$ as in our proof is a completely new result. Moreover, if the class $\cu^{\alz}$ is essentially small itself then one can take $\cu'=\cu$ in this reasoning.

2. Combining part \ref{iclassts} of our theorem with Proposition \ref{phop}(\ref{itpt}) (resp. Proposition \ref{phop}(\ref{itpw}))  we obtain a bijection between compactly generated $t$-structures (resp. weight structures) and those essentially small 
retraction-closed extension-closed subclasses  
of $\obj\cu^{\alz}$ that are also closed with respect to $[1]$ (resp. $[-1]$); this  generalizes Theorem 4.5 of ibid. to arbitrary triangulated categories having coproducts.

Moreover, we obtain the existence of a certain "join" operation on compactly generated  torsion theories; cf. Remark \ref{revenmorews} above. 

3. We also obtain that part \ref{iclasst} of our theorem generalizes Theorem A.9 of \cite{kellerw} where 
 stable derivator  categories were considered (similarly to the aforementioned results of \cite{postov}).

4. The question whether all  smashing weight structures  on a given compactly generated category $\cu$ are compactly generated is a certain weight structure version of  the (generalized) telescope conjecture (that is also sometimes called the smashing conjecture) for $\cu$; this question generalizes its "usual" stable version (see Proposition 3.4(4,5) of \cite{bvt}). 
It is well known (see the main result of \cite{kellerema})  that the answer to the shift-stable version of the  question is negative for a general $\cu$; hence this is only more so for our weight structure version. On the other hand, the answer to our question for $\cu=SH$ (the topological stable homotopy category)  is not clear.

5. The description of compact objects in $\lo$ provided by part \ref{iclass2} of our theorem is important for the continuity arguments in \cite{binfeff}.\end{rema}

\subsection{On well generated weight structures and torsion theories}\label{swgws}

Now we will prove that all smashing weight structures on well generated triangulated categories are perfectly generated (and also {\it strongly well generated}). 
 Unfortunately,  
 this will require several definitions and 
technical facts.

\begin{defi}\label{dbecomp}
 Let $\cu$ be a smashing triangulated category, and 
 $\be$ be a regular infinite cardinal (that is $\be$ cannot be presented as a sum of less than $\be$ cardinals that are less than $\be$), $\cp\subset \obj \cu$, and $\cocp$ is the closure of $\cp$ with respect to $\cu$-coproducts.
\begin{enumerate}

\item\label{idsmall}
 An object $M$ of $\cu$ is said to be {\it $\be$-small} if for any small family $N_i\in \obj \cu$ any morphism $M\to \coprod N_i$ factors through the coproduct of a subset of $\{N_i\}$ of cardinality less than $\be$.

\item\label{idcomp} We will say that an object $M$ of $\cu$ is {\it $\be$-compact} if it belongs to the maximal   perfect class of $\be$-small objects of $\cu$ (whose existence is immediate from Lemma \ref{lper}(\ref{il6})). 

 We will write $\cu^{\be}$ for the full subcategory of $\cu$ formed by $\be$-compact objects.

\item\label{idpg} 
 We  say that $\cu$ is $\be$-{\it well generated} (or just well generated) if  there exists a perfect {\bf set} of 
$\be$-small objects that generates $\cu$ as its own localizing subcategory.\footnote{Note that these objects will automatically be $\be$-compact; see the previous part of this definition.} 

\item\label{idclass} 
 A  class $\cpt$ of objects of $\cu$ is said to be {\it $\be$-coproductive} if it is closed with respect to $\cu$-coproducts of less than $\be$ objects.

\item\label{idchop} We will say that a torsion theory $s=(\lo',\ro')$ on a full triangulated subcategory $\cu'$ of $\cu$ is {\it $\be$-coproductive} if both $\obj \cu'$ and $\ro'$ are $\be$-coproductive.

\item\label{idapp} 
 A morphism $h\in \cu(M,N)$ (for some $M,N\in \obj\cu$) is said to be a {\it $\cp$-approximation} (of $N$) if $h$ is  $\cp$-epic (see Definition \ref{dpcl}(1)) and $M$ belongs to $\obj\cocp$.

\item\label{idcovf} We will say that $\cp$ is {\it contravariantly finite} (in $\cu$) if for any $N\in \obj \cu$ there exists its $\cp$-approximation.\footnote{Actually, the standard convention is to say that $\cocp$ is contravariantly finite if this condition is fulfilled; yet our version of this term is somewhat more convenient for the purposes of this section.}
\end{enumerate}
\end{defi}

\begin{rema}\label{rbecomp}
1. Our definition of $\be$-compact objects is equivalent to the one used in \cite{krauwg}.
Indeed, coproducts of less than $\be$ of  $\be$-small objects are obviously $\be$-small; thus the class $\obj \cu^\be$ is $\be$-coproductive. Hence the equivalence of definitions follows from Lemma 4 of ibid. 
Furthermore, Lemma 6 of ibid. states  that 
(both of) these definitions are equivalent to Definition 4.2.7 
 of \cite{neebook} if we assume in addition that $\cu^{\be}$ is an essentially small category. 

2. Now we recall some more basic properties of $\be$-compact objects in an $\al$-well generated category $\cu$ assuming that $\be\ge \al$ are regular cardinal numbers.

Theorem A of \cite{krauwg} yields immediately that $\cu^\be$ is an essentially  small triangulated subcategory of $\cu$. 

Moreover, the union of  $\cu^{\gamma}$ for $\gamma$ running through all regular cardinals ($\ge \al$) equals $\cu$ (see the corollary 
 at the end of ibid. or Proposition 8.4.2 of \cite{neebook}). 

3. Lastly, we recall a part of \cite[Lemma 4]{krauwg}. 
For any $\be$-coproductive essentially small  perfect class $\cp$ of $\be$-small objects of a triangulated category $\cu$ (that has coproducts) it says the following:
for any $P\in \cp$ and any set of  $N_i\in \obj \cu$ any morphism $P\to \coprod N_i$ factors through the coproduct of some $\cu$-morphisms $M_i\to N_i$ with $M_i\in \cp$. 
\end{rema}

Let as now prove a collection of statements on smashing torsion theories; part III of the following theorem is dedicated to weight structures and appears to be its most interesting part.

\begin{theo}\label{twgws}
Let $s=(\lo,\ro)$ be a smashing torsion theory on $\cu$,  
 and $\cp\subset \obj \cu$. 

I. Consider the class $J$ of $\cu$-morphisms characterized by the following condition: $h\in \cu(M,N)$ (for $M,N\in \obj \cu$) belongs to $J$ whenever for any chain  of morphisms $ L_sP\stackrel{a_P}{\longrightarrow} P\stackrel{g}{\to}M \stackrel{h}{\to}N$ its composition is zero if $P\in \cp$ and $a_P$ is an $s$-decomposition morphism (see Definition \ref{dhop}(\ref{ittd})).

Then the following statements are valid.

\begin{enumerate}
\item\label{indep}  The class $J$ will not change if we will fix $a_P$ for any $P\in \cp$ in this definition.

\item\label{icontraf} Assume that $\cp$ is  contravariantly finite and $s$ is smashing. Then $h$ 
belongs to $J$ if and only if there exists a
$\cp$-approximation morphism $AM\stackrel{g}{\to} M$ and an $s$-decomposition morphism $a_{AM}: L_sAM\to AM$ such that 
$h\circ g \circ a_{AM}=0$. Moreover, the latter is equivalent to the vanishing of all compositions of this sort.

\item\label{icoprcl}
 Assume that $\cp$ is  contravariantly finite and  
 perfect, 
and $s$ is smashing. Then the class $J$ is closed with respect to 
 coproducts.

\item\label{ilscp}
 Assume that for any $P\in \cp$ there exists a choice of $ L_sP\in \cp$; denote the class of these choices by $\lscp$. Then $J$ coincides with the class of $\lscp$-null morphisms.

\item\label{ilscper} Assume in addition (to the previous assumption) that  $\cp$ is a 
 perfect   contravariantly finite  class 
and $s$ is smashing. Then $\lscp$ is a 
 perfect   contravariantly finite class as well.

\item\label{ilscperw}
 Assume in addition that $s$ is weighted (see Proposition \ref{phop}(\ref{itpw})); suppose  that the class $\cp$ is essentially small, equals $\cp[1]$,   and generates $\cu$ as its own localizing subcategory. Then  the class $L_s\cp$  generates $s$ and $\lo$ is the big hull of $L_s\cp$; thus $s$ is perfectly generated in the sense of Remark \ref{rigid}(3).
\end{enumerate}

II. For a regular cardinal $\be$ let  $s'=(\lo',\ro')$ be a $\be$-coproductive torsion theory on a full triangulated category $\cupr$ of $\cu$ such that $\obj \cupr$ is a  perfect essentially small 
class of  $\be$-small objects.  Then $\lo'$ is  perfect as well. 

Moreover, if $s'$ is weighted in $\cupr$ then $\lo'$ generates a weighted smashing torsion theory on  $\cu$.

III. Assume in addition that $\cu$ is $\al$-well generated for some regular cardinal $\al$ (see Definition \ref{dbecomp}(\ref{idpg})), and that  $s$ is smashing. 

1. Assume that $s$ restricts (see Definition \ref{dhop}(\ref{ittr})) to $\cu^{\be}$  for a regular cardinal $\be \ge \al$. Then $\lo\cap \obj \cu^\be$ is an essentially small  perfect class.

2. If $s$ is weighted then it restricts to $\cu^{\be}$ for all large enough regular $\be\ge \al$; being more precise, it suffices to assume the existence of $L_sM\in \obj \cu^{\be}$ for all $M\in \obj \cu^{\al}$. Moreover, the class $\lo\cap \obj \cu^\be$ is perfect and generates $s$ for any $\be$ that satisfies this inequality. 
\end{theo}
\begin{proof}

I.\ref{indep}. It suffices to note that any $s$-decomposition morphism for $M$ factors through any other one according to Proposition \ref{phop}(\ref{itp7}).

\ref{icontraf}. We fix $h$ (along with $M$ and $N$).  

The definition of approximations along with Proposition \ref{phop}(\ref{itp7}) implies that any composition  $ L_sP\stackrel{a_P}{\longrightarrow} P\stackrel{g}{\to}M$ as in the definition of $J$ factors through the composition morphism  $L_sAM\to M$. Hence if the composition $L_sAM\to N$ is zero then $h\in J$.

Conversely, assume that $h\in J$.
Since  $\cp$ is  contravariantly finite, we can choose  a $\cp$-approximation morphism $g\in \cu(AM, M)$. 
Present $AM$ as a coproduct of some $P_i\in \cp$; choose some $s$-decomposition morphisms $L_sP_i\stackrel{a_{P_i}}{\longrightarrow} P_i$. Since $s$ is smashing,  the morphism $a_{AM}^0=\coprod a_{P_i}$ is an $s$-decomposition one as well according to Proposition \ref{phop}(\ref{itp4}). Since $h \circ g\circ  a_{P_i}=0$ for all $i$, we also have $h \circ g\circ  a_{AM}^0=0$. Lastly,  any other choice of $a_{AM}$ factors through $a_{AM}^0$ (by Proposition \ref{phop}(\ref{itp7}); cf. the proof of assertion I.\ref{indep}); this gives the "moreover" part of our assertion.

\ref{icoprcl}. This is an easy consequence of the previous assertion. Note firstly that Lemma \ref{lper}(\ref{il2}) (along with the dual to Proposition \ref{pcgt}(I.3)) 
  easily implies that for any choices of $\cp$-approximations $AM_i\to M_i$ their coproduct is a  $\cp$-approximation of $\coprod M_i$.
 The assertion follows easily since the coproduct of any choices of $L_sAM_i\to AM_i$ of $s$-decomposition morphisms is an  $s$-decomposition morphism as well (according to Proposition \ref{phop}(\ref{itp4})); thus it remains to apply assertion I.\ref{icontraf}.

\ref{ilscp}. 
 Assertion I.\ref{indep}  implies that any $\lscp$-null morphism belongs to $J$. The converse implication is immediate from $\lscp\subset \cp$.

\ref{ilscper}. This is an obvious combination of the previous two assertions.

\ref{ilscperw}. Since $\lo$ contains $\lscp$, it also contains its big hull (see Lemma \ref{lbes}(\ref{iseses}, \ref{isesperp})). Thus it suffices to verify the converse inclusion. 

Now, since $\cp$ is essentially small,  perfect, and $\cp= \cp[1]$, the big hull of $\cp$ along with $\cp^\perp$ give a (weighted) torsion theory according to Theorem \ref{tpgws}(1). 
   Since $\cp$ generates $\cu$ as its own localizing subcategory, 
$\cp^\perp=\ns$ (see Proposition 8.4.1 of \cite{neebook}); 
 thus any object of $\cu$ belongs to the big hull of $\cp$.  

Now let $P$ belong $\lo$. As we have just proved, it is a retract of some $Y$ that belongs to the naive big hull of $\cp$. We present $Y$ as $\hinli Y_i$ so that $Y_0$ and cones $Z_{i+1}$ of the connecting morphisms $f_i$ belong to 
$\cocp$. 
 Hence  we can choose $L_sZ_{i+1}$ to belong to $ {\underline{\coprod} \lscp}$; we fix these choices.

Applying Proposition \ref{pbw}(\ref{iwdext}) inductively we obtain that for all $i\ge 0$ there exist  connecting morphisms $l_i:L_sY_i\to L_sY_{i+1}$ between certain choices of $L_sY_*$ such that such that the corresponding squares commute and $\co(l_i)\cong  L_sZ_{i+1}$.

Now we consider the commutative square $$ \begin{CD}
 \coprod L_sY_i@>{L_sa}>>\coprod L_sY_i\\
@VV{\coprod a_{Y_i}}V@VV{\coprod a_{Y_i}}V \\
\coprod Y_i @>{a}>>\coprod Y_i
\end{CD}$$
where $a$ is the morphism $\oplus \id_{Y_i}\bigoplus \oplus (-f_i)$ (cf. Definition \ref{dcoulim}) and
 $L_sa=\oplus \id_{L_sY_i}\bigoplus \oplus (-l_i): \coprod L_sY_i\to \coprod L_sY_i$ is the morphism corresponding to $\hinli L_sY_i$. 
 According to Proposition 1.1.11 of \cite{bbd}, we can complete it to a commutative diagram 
\begin{equation}\label{ely}\begin{CD}
\coprod L_sY_i @>{L_sa}>>\coprod L_sY_i@>{}>> L_sY@ >{}>>\coprod L_sY_i[1] \\
 @VV{\coprod a_{Y_i}}V@VV{\coprod a_{Y_i}}V @VV{}V@VV{\coprod a_{Y_i}[1]}V\\
\coprod Y_i @>{a}>>\coprod Y_i@>{}>> Y@>{}>> \coprod Y_i[1]\\
@VV{}V @VV{}V @VV{}V@VV{}V \\
\coprod R_sY_i @>{}>>\coprod R_sY_i@>{}>> R_sY @>{}>> \coprod R_sY_i[1] \\
\end{CD}\end{equation}
whose rows and columns are distinguished triangles. Then $L_sY$ is a homotopy colimit of $LY_i$ (with respect  to $l_i$) by definition; thus it belongs to the naive big hull of $\lscp$.  
Next,  the 
bottom row  of (\ref{ely}) gives  $R_sY\in \ro$ (since 
$\coprod RY_i\in  \ro$). 
Thus the 
 third column of our diagram is an $s$-decomposition  of $Y$. Hence applying Proposition \ref{phop}(\ref{itp7p}) we obtain that $P$ belongs to the  big hull of 
$ \lscp$.


II. Let $f_i\in \cu(N_i,Q_i)$ for $i\in J$ be a set of $\lo'$-null morphisms; for $N=\coprod N_i$, $f=\coprod f_i$, and $P\in \lo'$ we should check that the composition of any $e\in \cu(P,N)$ with $f$ vanishes. The $\be$-smallness of $P$ allows us to assume that $J$ contains less than $\be$ elements.

Next, Remark \ref{rbecomp}(3) gives a factorization of $e$ through the coproduct of some $h_i\in \cu(M_i,N_i)$ with $M_i\in \obj \cupr$. We choose some $s'$-decompositions $L_i\to M_i\to R_i\to L_i[1]$ of $M_i$. Our assumptions easily imply that $\coprod L_i\to \coprod M_i\to \coprod R_i$ is an $s'$-decomposition of $\coprod M_i$ (cf. Proposition \ref{phop}(\ref{itp4})). Hence part  \ref{itp7} of the proposition implies that $e$ factors through the coproduct $g$ of the corresponding morphisms $L_i\to N_i$. Now, since $f_i$ are $\lo'$-null and $L_i\in \lo'$ then $f\circ g=0$; hence $f\circ e=0$ as well.

Lastly, if $s'$ is weighted then $\lo'$ contains $\lo'[-1]$. 
  Since $\lo'$ is also essentially small it remains to apply Theorem \ref{tpgws}(1,3). 

III. 
For a regular cardinal $\be\ge \al$ we take $\cp=  \obj \cu^\be$. This is clearly a perfect essentially small class that generates  $\cu$ as its own localizing subcategory; 
 we also have $\cp= \cp[1]$. 

 To prove assertion III.1 it suffices to note that $\lo\cap \obj \cu^\be$ is a possible choice of $L_s\cp$ (in the notation of assertion I) and apply assertion I.\ref{ilscper}. 

Next, assertion I.\ref{ilscperw} 
 implies that to prove assertion III.2 it suffices to verify that $s$ restricts to $\cu^{\be}$ for all large enough regular $\be\ge \al$.

Now we choose some $L_sM$ for all $M\in \obj \cu^\al$, and take   a regular cardinal $\al'$ such that all elements of $L_s\cp$ belong to $\cu^{\al'}$ (see Remark \ref{rbecomp}(2)). Then for any regular $\be\ge {\al'}$ the torsion theory  $s$ restricts to $\cu^{\be}$, since the corresponding weight decompositions exist according to Proposition 2.3.4(3) of \cite{bwcp}.
\end{proof}

\begin{rema}\label{rtkrau}
1. Our theorem suggests that it makes sense to define (at least) two distinct notions of $\be$-well generatedness for smashing torsion theories and weight structures  in an $\al$-well generated category $\cu$. One may say that $s$ is {\it weakly $\be$-well generated} for some regular $\be\ge \al$ if it is generated by a perfect set 
 of $\be$-small objects.
$s$ is {\it strongly  $\be$-well generated} if in addition to this condition, $s$ restricts to $\cu^\be$.

Clearly, compactly generated torsion theories (see Definition \ref{dhop}(\ref{ittg})) 
 are precisely the weakly $\alz$-well generated ones (since any set of compact objects is perfect; see Lemma \ref{lper}(\ref{il7})). Hence  our two notions of $\be$-well generatedness are not equivalent (already) in the case $\al=\be=\alz$; this claim follows from \cite[Theorems  4.15, 5.5]{postov} (cf. also Corollary 5.6 of ibid.) where (both weakly and strongly $\alz$-well generated) weight structures on $\cu=D(\modd-R)$ were considered in detail.

Moreover, for $k$ being a field of cardinality $\gamma$ the main subject of \cite{bgn} gives the following example: the opposite (see Proposition \ref{pbw}(\ref{idual})) to (any version of) the Gersten weight structure over $k$ (on the category $\cu$ that is opposite  to the corresponding category of {\it motivic pro-spectra}; note that $\cu$ is compactly generated) is 
weakly  $\alz$-well generated (by definition) and it does not restrict to the subcategory of $\be$-compact objects for any $\be\le \gamma$. On the other hand, this example is "as bad is possible" 
  for weakly  $\alz$-well generated weight structures in the following sense: combining the arguments used the proof of part III.2 of our theorem with that for Theorem \ref{tclass} one can easily verify  that any $\alz$-well generated weight structure is $\al$-well generated whenever the set of  (all) isomorphism classes of morphisms 
in the subcategory $\cu^\alz$  of compact objects of $\cu$ is of cardinality less than $\al$. 

Note also that general strongly $\alz$-well generated weight structures were treated in detail in \S3.3 of \cite{bvtr}. 

2. Obviously the join (see Remark \ref{revenmorews} and Corollary \ref{cwftw}(2)) of any set of  weakly $\be$-well generated weight structures  is    weakly $\be$-well generated; thus we obtain a filtration (respected by joins) on the "join monoid" of weight structures. 
 The natural analogue of this fact  for strongly $\be$-well generated weight structures is probably wrong. Indeed, it is rather difficult to believe that for a general compactly generated category $\cu$ the class of weight structures on the subcategory $\cu^{\alz}$ 
 would be closed with respect to joins; note that  joining compactly generated weight structures $w_i$ on $\cu$ corresponds to intersecting the classes $\cu_{w_i\ge 0}\cap \obj \cu^{\alz}$.

On the other hand, Corollary 4.7 of \cite{krause}  suggests that the filtration of   the class of smashing weight structures  by the sets of weakly $\be$-well generated ones (for $\be$ running through regular cardinals) may be "quite short".  

3. According to part III.2 of our theorem, any weight structure on a well generated $\cu$ is strongly $\be$-well generated for $\be$ being large enough.  Combining this part of the theorem with its part II 
 we also obtain a bijection between strongly  $\be$-well generated weight structures on $\cu$ and $\be$-coproductive weight structures on $\cu^\be$. 
Note that (even) the restrictions of these results to compactly generated categories appear to be quite interesting. 

4. For $\cu$ as above and 
 a weakly $\be$-well generated weight structure $w$ on it one can easily establish a natural weight structure analogue of \cite[Theorem B]{krauwg} that will 
"estimate the size" of an element $M$ of $\cu_{w\le 0}$ in terms of the cardinalities of $\cu(P,M)$ for $P$ running through $\be$-compact elements of  $\cu_{w\le 0}$ (modifying the proof of loc. cit. that is closely related to our proof of  Theorem \ref{tpgws}). Moreover, this result should generalize loc. cit. Note also that there is a "uniform" estimate of this sort that only depends on $\cu$ (and does not depend on $w$).
 This argument should also yield that a weakly $\be$-well generated weight structure is always strongly $\be'$-well generated for a regular cardinal $\be'$ that can be described explicitly.

Moreover, similar arguments can possibly yield that any smashing weight structure on a perfectly generated triangulated category $\cu$ is perfectly generated (cf. Theorem \ref{twgws}(III.2)).

5. Our understanding of "general" well generated torsion theories is much worse than the one of (well generated) weight structures. In particular, the author does not know which properties of weight structures proved in this section can be carried over to $t$-structures.
\end{rema}

\appendix
\section{
 More 
  adjacent weight and $t$-structures}\label{sdeg}

To generalize Theorem \ref{tsymt}(II) and discuss the conditions of the corresponding result 
we 
  need some of the theory of {\it Bousfield localizations} (in the 
   terminology of \cite[\S9]{neebook}; however, we take some  notation from \cite{versga45} and also cite \cite{krauloc}). 

\begin{llem}\label{lwdeg}
  Assume that  there exists an exact right adjoint $i^*$ to a strictly full exact embedding $i:\cu'\to \cu$. 
   Denote by $\du$ the triangulated subcategory of $\cu'$-local objects in $\cu$ (see Lemma \ref{lcloc}).

  \begin{enumerate}
  
  
    \item\label{ibld}  
    $\perpp\obj \du=\cu'$ and 
  there exists a left adjoint  
   ${}^*Q\perpp$ to the localization $Q\perpp:\cu\to \eu=\cu/\du$.

    \item\label{ile3} 
      The functor  ${}^*Q\perpp$ is fully faithful, and the  composition $i\circ i^*$ is isomorphic to  ${}^*Q\perpp \circ  Q\perpp$.
     
      \item\label{ile4} Denote by  $\du'$ the full subcategory of 
       $\du$-local objects of $\cu$, that is,  
  $\obj \du'=\obj \du\perpp$. Then the restriction of the obvious transformation of bi-functors $\cu(-,-)\to \eu(Q\perpp(-),Q\perpp(-))$ to $\cu\opp\times \du'$ is an isomorphism. 
  Consequently, the restriction of $i^*$ to $\du'$ is fully faithful.
     
     \end{enumerate}

  \end{llem}
\begin{proof}

  Assertion \ref{ibld} easily follows from Proposition 4.9.1 of  \cite{krauloc}; see conditions (2) and (6) in it.
  
   \ref{ile3}. 
   Assertion \ref{ibld} allows to apply to loc. cit. to the embedding $i\perpp{}\opp:\du\opp\to \cu\opp$. Consequently, 
     the assertion easily follows from loc. cit., 
     see condition (1) in loc. cit. and Corollaries  2.4.2 and  2.5.3 of ibid.\footnote{Alternatively, one may apply Proposition I.2.6.7 of \cite{versga45} to obtain the second half of the assertion.}

   
   \ref{ile4}.  
   The first part of the assertion immediately follows from 
     Proposition I.2.5.3 of \cite{versga45} 
     (and it is essentially Lemma  9.1.5 of \cite{neebook}). 
   Combining this statement with the previous assertion we deduce the full embedding statement easily. 
  
\end{proof}

\begin{rrema}
Note also that the localization $\eu=\cu/\du$ 
is a locally small category; see Theorem 9.1.16 and Remark 9.1.17 of \cite{neebook} or condition (5) in   \cite[Proposition 4.9.1]{krauloc}. \end{rrema}

Now we generalize Theorem \ref{tsymt}(II). 

\begin{theore}\label{tpsymt} 
Assume that $\cu$ is both smashing and cosmashing, $\cu'$ is its localizing subcategory generated by some class $\cp\subset \obj \cu$, and there exists a right adjoint $i^*$ to the embedding $i:\cu'\to \cu$. Suppose also that $\cp$ is {\it weakly symmetric} to some {\bf set} of objects of $\cu$, that is, $\cp\perpp=\perpp \cp'$, and $\cp'$ is perfect in the category $\cu\opp$.  

Then the following statements are valid.

1. $\cp$ is weakly symmetric to the set $i^*(\cp')$ in $\cu'$. Moreover,  
 $i^*(\cp')$ is perfect in the category $\cupr{}\opp$ and generates it as the localizing subcategory of $\cu\opp$; consequently, $\cupr{}\opp$ satisfies the Brown representability condition. 
 
 2. Take $\du$ and $\du'$ as in Lemma \ref{lwdeg}, that is, $\obj \du=\obj \cu'{}\perpp$ and  $\obj \du'=\obj \du\perpp$.
  Then both 
   $\du$ and $\du'$ are triangulated subcategories of $\cu$, and 
  $i^*$ restricts to an equivalence of $\du'$  to $\cu'$. 
   Furthermore, the set $\cp'$
  generates $\du'{}^{op}$ as the localizing subcategory of $\cu\opp$, and  the  embedding $i':\du'\to \cu$ 
  possesses an exact left adjoint. 

3. $\cp$ generates some $t$-structure $t$ on $\cu$ and 
 there exists a cosmashing  weight structure $w$  that is right adjacent to $t$. Moreover, the opposite weight structure $w\opp$ on $\cu\opp$ (see Proposition \ref{pbw}(\ref{idual})) is generated by $\cp'$.  

4. $\hrt$  has 
an injective cogenerator, and 
 the functor $H^t$ restricts to an equivalence of $\hw$ 
  with the subcategory of  injective objects of $\hrt$.
\end{theore}
\begin{proof}
1. Since 
  $\cp'$ is perfect in  category $\cu\opp$, the morphism class $\cpcn$ (see Definition \ref{dsym}(\ref{isym})) is closed with respect to $\cu$-products. Next, the functor $i^*$ respects products since it is  right adjoint to $i$. Moreover,  $i^*$ is essentially identical on $\cu'$; thus $\cu'$ is cosmashing (both in $\cu$ and as a triangulated category) and  
 the class $i^*(\cp')-\mathbf{conull}_{\cu'}  $ is closed with respect to $\cu'$-products. Hence $i^*(\cp')$ is perfect in the category $\cupr{}\opp$ indeed.
 
  Furthermore, the adjunction of $i$ to $i^*$ clearly implies that   $\cp$ is $\cu'$-weakly symmetric to $i^*(\cp')$. 
 Moreover, the class $(\cup_{i\in \z} \cp[i])^{\perp_{\cu'}}$ is zero by Lemma \ref{lper}(\ref{il4}); hence $(\cup_{i\in \z} i^*(\cp')[i])^{\perp_{\cu'{}\opp}}=\ns$ as well.
 Therefore Lemma \ref{lper}(\ref{il4})  also implies that  $i^*(\cp')$ generates $\cu'{}\opp$ as its own localizing subcategory, and this gives the Brown representability condition for $\cupr{}\opp$.


2. 
 Lemma \ref{lcloc} implies that  $\du$ and $\du'$ are triangulated subcategories of $\cu$. 
Moreover, 
 $\du'$ is a cosmashing subcategory of $\cu$;  hence 
 $i'$ respects products.  

Now we study  the restriction $i^*_{\du'}$  of $i^*$ to $\du'$. $i^*_{\du'}$  respects $\du'$-products 
 and it is fully faithful by Lemma \ref{lwdeg}(\ref{ile4}). Since  
   $\du'$ contains $\cp'$ and  $i^*_{\du'}$  sends $\du'{}\opp$ into a category generated by $i^*(\cp')$ as its own localizing subcategory,
  $\du'^{op}$   is perfectly generated by $\cp'$ indeed. Applying Lemma \ref{lper}(\ref{il4}) to the set $\cp'$ of objects of $\cu\opp$ we also obtain that $i'$ possesses an exact left adjoint.

3. We argue similarly to the proof of  Theorem \ref{tsymt}(II). 
Denote by $w^{op}$ (resp. $w'{}\opp$) the weight structure 
  on 
   $\cu\opp$ generated by  $\cp'$  (resp. by the weight structure generated by $i^*(\cp)$ in $\cu'$); see 
    Theorem \ref{tpgws}.
The weight structures $w\opp$ and $w'{}\opp$  are smashing according to Theorem \ref{tpgws}(3).  Hence the corresponding (opposite) weight structure $w$ on $\cu$ 
  is cosmashing and  the class $\cu_{w=0}$ is cosmashing in $\cu$. 

Moreover, by Proposition \ref{pbw}(\ref{iexadj}) there exists a $t$-structure $t'{}\opp$ on $\cu'{}\opp$ that is right adjacent to $w'{}\opp$. We set $t'$ to be the corresponding $t$-structure on $\cu'$; see Remark \ref{rtst}(1). Then $$\cu'_{t'\le 0}=\cu'_{w'{}\opp\ge 0}={}^{\perp_{\cu'}}(\cup_{i>0}i^*(\cp')[i])=(\cup_{i>0}\cp[i])^{\perp_{\cu'}};$$
 see Definition \ref{dwso}(\ref{idadj}). 
  Thus $t'$ is generated by $\cp$. Applying Proposition \ref{prtst}(\ref{itdeg}) we obtain the existence of a $t$-structure $t$ on $\cu$ that is generated by $\cp$ (as well). $t$ is left adjacent to $w$ by Proposition 4.4(3) of \cite{bvt}  (along with 
    Proposition 3.2(1,2) of ibid.).  

   4.  
   Proposition \ref{pbw}(\ref{injadj},\ref{igenws}) implies that  $H^t$ gives an equivalence of  $\hw$ 
    with the subcategory of injective objects of $\hrt$.  Moreover, this proposition also says that $\hrt$ has enough injectives,
    whereas Theorem \ref{tpgws}(3) gives  the  existence of  $I_w\in \cu_{w=0}$  such that any 
element of  $\cu_{w=0}$ is a retract of a product of copies of $I_w$. 
   \end{proof}

\begin{rrema}
1. Let us now relate the theorem above to Theorem \ref{tsymt}(II) and other statements related to this subject.

Note that $\cp$ is weakly symmetric to $\cp'$ if it is symmetric to it; see 
 Lemma \ref{lper}(\ref{il1}). Thus if $\cu$ and smashing and cosmashing and $\cp$ and $\cp'$ are symmetric sets in it then all the assumptions of Theorem \ref{tpsymt} are fulfilled; see 
 Theorem \ref{tsymt}(I.1) and Lemma \ref{lper}(\ref{il4}). 
  Hence our theorem essentially implies Theorem 4.3.8 of \cite{bpure}.\footnote{Actually,  in loc. cit. it is not assumed that $\cp$ is a set. However, it appears that to make the proof work one has either to add this condition or just suppose that the (adjoint) functor $i^*:\cu\to \cu'$ exists.}
Consequently, our theorem generalizes Theorem 3.6(1) of \cite{modoiwcc}.


We recall that the argument of  \cite{modoiwcc} required   $\cu$ to be a "strong stable derivator" triangulated category (cf. the footnote to Theorem \ref{tgroth} above), whereas the elements of $\cp$ (if we use our notation) were required to be {\it weakly  compact} in a certain sense
 (see \S1.5 of ibid.). Note however that these two extra assumptions do not appear to be really restrictive, and the arguments of ibid. are quite diffierent from our ones.

2. 
 Combining the aforementioned  Proposition 4.9.1 of  \cite{krauloc} with Theorem \ref{tpsymt}(1,2) we obtain that   Bousfield localisation 
 functors exist both for the pair  $\du\subset \cu$ and for $\du\opp\subset \cu\opp$ (under the assumptions of our theorem); see Definitions 9.1.1 and Definitions 9.2.1 of \cite{neebook}. 
Thus we are in the {\it situation of the six gluing functors} (see loc. cit.);  alternative terms are {\it gluing datum}  and {\it recollement} (cf. \S1.4.3 and Remark 1.4.8 of \cite{bbd}).  

Now let us make a few observations related to this notion. If an exact embedding $i\perpp:\du\to \cu$ is a part of a gluing datum, $\cu'$ and $\du'$ are the (triangulated) subcategories of $\cu$ whose object classes equal $\perpp \obj \du$ and $ \obj \du\perpp$, respectively, then there exists a right adjoint $i^*$ to the embedding $i:\cu'\to \du$ 
 and $i^*$ restricts to an equivalence $i^*_{\du'}:\du'\to \cu'$ (see \S1.4.6(b) of ibid.). Moreover, 
  Lemma \ref{lwdeg} easily implies the following: if a class $\cp$ is (weakly) symmetric to 
  some $\cp''\subset \obj \cu'$ in $\cu'$ then  $\cp$ (weakly) symmetric to the essentially small class $\cp'$ in $\cu$, where  $\cp'=i^*_{\du'}{}\ob(\cp')$. 

Consequently, one may apply some sort of Brown-Comenetz dualily (see \S\ref{sadjts}) to find some symmetric sets in $\cu'$ and "extend them" to $\cu$ as above whenever $\cu'$ is a smashing subcategory of $\cu$ that satisfies the Brown representability condition 
and the embedding $i:\cu'\to \cu$ extends to a gluing datum. This is the most general method of    constructing symmetric sets in smashing triangulated categories 
 currently known to the author. 
\end{rrema}

\end{document}